\documentclass[11pt,letterpaper]{amsart}
\usepackage{graphicx} % Required for inserting images
\usepackage[margin=30truemm]{geometry}

\usepackage{amsmath,amssymb,amsthm,amscd,enumitem,url,hyperref,mleftright,bm,calc}

\numberwithin{equation}{section}

\newcommand{\bigast}{
  \mathop{
    \vphantom{\oplus} 
    \mathchoice
      {\vcenter{\hbox{\resizebox{\widthof{$\displaystyle\oplus$}}{!}{$\ast$}}}}
      {\vcenter{\hbox{\resizebox{\widthof{$\oplus$}}{!}{$\ast$}}}}
      {\vcenter{\hbox{\resizebox{\widthof{$\scriptstyle\oplus$}}{!}{$\ast$}}}}
      {\vcenter{\hbox{\resizebox{\widthof{$\scriptscriptstyle\oplus$}}{!}{$\ast$}}}}
  }\displaylimits 
}

\theoremstyle{plain}
\newtheorem{thm}{Theorem}[section]
\newtheorem{cor}[thm]{Corollary}

\newtheorem{prop}[thm]{Proposition}
\newtheorem{lem}[thm]{Lemma}

\theoremstyle{definition}
\newtheorem{dfn}[thm]{Definition}
\newtheorem{nta}[thm]{Notation}

\theoremstyle{remark}
\newtheorem{rmk}[thm]{Remark}
\newtheorem{exm}[thm]{Example}

\title{Tracks on planar complexes and soficity}
\author{Hiroki Ishikura}
\date{}

\begin{document}

\begin{abstract}
We show that every probability-measure-preserving equivalence relation generated by a locally-finite Borel graph with planar connected components is sofic in the sense of Elek--Lippner. In particular, every unimodular random planar graph is sofic. This removes the additional assumptions in the works of Angel--Hutchcroft--Nachmias--Ray and Tim\'{a}r on the soficity of unimodular random planar maps and graphs. To prove this, we investigate Borel simplicial complexes with planar components and approximate them by treeable covering spaces. To construct these coverings, we use a canonical family of tracks on planar simplicial complexes introduced by Dunwoody.
\end{abstract}

\maketitle

\section{Introduction}\label{sec1}

Soficity of countable groups was originally introduced by Gromov \cite{G} and Weiss \cite{W}.
This notion was later generalized by Aldous and Lyons \cite{AL} to unimodular random graphs, which can be viewed as a generalization of Cayley graphs of finitely generated groups. That is, a unimodular random rooted graph $(\mathbb{G},o)$ is said to be \textit{sofic} if it is a Benjamini--Schramm limit of finite graphs with uniform roots. The Aldous--Lyons conjecture, which states that every unimodular random graph is sofic, was recently refuted by \cite{BCLV,BCV}. However, we show that all unimodular random graphs satisfying a combinatorial condition, namely planarity, are indeed sofic.

\begin{thm}\label{thm1}
Every unimodular random planar graph is sofic.
\end{thm}

Here, a locally-finite connected graph is said to be planar if it admits a topological embedding into the Euclidean plane. This property is equivalent to admitting an embedding into the $2$-sphere. By Kuratowski's theorem, it is also equivalent to not containing either the complete graph $K_5$ or the complete bipartite graph $K_{3,3}$ as a minor. The above theorem removes the assumption on the number of ends in \cite[Theorem 1]{T}.

It is known that every unimodular random graph $(\mathbb{G},o)$ is realized as a graphing, or a measure-preserving locally-finite Borel graph $G$ on a standard probability space $(X,\mu)$. That is, $(\mathbb{G},o)$ coincides with $(G_x,x)$ as a random variable, where $x\in X$ is a $\mu$-random point and $G_x$ is the connected component of $G$ containing $x$. If the equivalence relation $\mathcal{R}_G$ on $(X,\mu)$ generated by $G$ is \textit{$\mu$-sofic} in the sense of \cite{EL}, then the unimodular random graph $(\mathbb{G},o)$ is also sofic. Therefore, Theorem \ref{thm1} is a direct consequence of the following:

\begin{thm}[Corollary \ref{cor3}]\label{thm3}
Let $G$ be a probability-measure-preserving locally-finite planar graph on $(X,\mu)$. Then the equivalence relation $\mathcal{R}_G$ on $(X,\mu)$ generated by $G$ is $\mu$-sofic.
\end{thm}

Here we say that a Borel graph is planar if every connected component is planar.

In \cite{CGMT}, the term ``Borel planar graph'' is used in a much stronger sense than here. For example, in their definition, every component of the Borel graph must have a $2$-basis, which is not true for general planar graphs. Therefore, their work does not directly imply the above theorem to the best of our understanding.

\subsection{On treeability}
Throughout this article, all Borel graphs are assumed to be locally finite. Fix a standard probability space $(X,\mu)$. The term \textit{pmp} means probability-measure-preserving.

Unimodular random \textit{trees} are known to be sofic \cite{B,E,BLS}. This is generalized by \cite[Theorem 4]{EL}, which states that every pmp $\mu$-treeable equivalence relation on $(X,\mu)$ is $\mu$-sofic. Here, a countable Borel equivalence relation $\mathcal{R}$ on $(X,\mu)$ is $\mu$-\textit{treeable} if there exist a $\mu$-conull subset $X'\subset X$ and an acyclic Borel graph $T$ on $X'$ such that $\mathcal{R}|_{X'}=\mathcal{R}_T$.

A unimodular random graph $(\mathbb{G},o)$ is said to be \textit{treeable} if it is realized as a pmp graph $G$ on $(X,\mu)$ such that $\mathcal{R}_G$ is $\mu$-treeable. This is equivalent to the existence of a ``rewiring tree" for the random graph $\mathbb{G}$, i.e., a random tree $\mathbb{T}$ on the vertex set of $\mathbb{G}$ satisfying a certain invariance.
This condition implies the soficity of $(\mathbb{G},o)$.
Tim\'{a}r actually proves the treeability of unimodular random one-ended planar graphs in \cite{T}. This work relies on the property of unimodular random planar maps investigated by Angel--Hutchcroft--Nachmias--Ray \cite{AHNR}. Since they use the free uniform spanning forest to construct a rewiring tree, they need to add extra randomness to the given unimodular random graph. Therefore, it has been unknown whether the relation $\mathcal{R}_G$ is $\mu$-treeable for every pmp one-ended planar graph $G$ on $(X,\mu)$.

The positive answer to the above question is given by Conley--Gaboriau--Marks--Tucker-Drob \cite{CGMT} and Jard\'{o}n-S\'{a}nchez \cite{J} (even when the Borel graph is not pmp). More generally:

\begin{thm}[{\cite[Theorem 5.2]{J}}] \label{thmJ1}
If $G$ is a planar Borel graph on $(X,\mu)$ such that every component is accessible, then $\mathcal{R}_G$ is $\mu$-treeable.
\end{thm}

Here, a locally-finite connected graph is said to be accessible if there exists a positive integer $n$ such that any two distinct ends can be separated by removing at most $n$ edges. In particular, one-ended graphs are accessible.

It remains an open question whether Theorem \ref{thmJ1} holds without assuming the accessibility of the components. Indeed, the proof of this theorem relies on the existence of a tree decomposition of a Borel graph into Borel graphs with one-ended components. Although it is hard to give a concrete example of a planar Borel graph that does not admit such a decomposition, there is no general way to do this process for Borel graphs with inaccessible components. To deal with this issue, we prove Theorem \ref{thm3} by approximating any planar Borel graph by its ``covering spaces" for which treeability can be established.

\subsection{Proof strategy}
In practice, we will mainly work with Borel complexes rather than Borel graphs.

A \textit{Borel complex} on $X$ is a locally finite simplicial complex such that the set of vertices is $X$ and the set of simplices is a Borel subset of the set of finite subsets of $X$. Then let $\mathcal{R}_{\Sigma}$ denote the connected relation of $\Sigma$ on $X$. We say that $\Sigma$ is a \textit{pmp complex} if $\mathcal{R}_{\Sigma}$ is pmp on $(X,\mu)$.

\begin{dfn}
A \textit{planar Borel complex} is a Borel complex such that every component admits a combinatorial embedding into the $2$-dimensional sphere. If moreover every component is homeomorphic to a planar surface with boundary, we say that the planar Borel complex has no \textit{singular vertices}.
\end{dfn}

We reduce the proof of Theorem \ref{thm3} to showing that every pmp planar complex without singular vertices is sofic.
Now the following is our main technical result. See Section \ref{sec2} for the definition of terminology.

\begin{thm}[Theorem \ref{thm2}]\label{thm2'}
Let $(X,\mu,\Sigma)$ be a pmp bounded-degree planar complex without singular vertices. Then there exist a sequence of pmp planar complexes $(X_n,\mu_n,\Sigma_n)$ such that $\mathcal{R}_{\Sigma_n}$ is $\mu_n$-treeable for every $n$, and a sequence of covering factor maps $\pi_n:(X_n,\mu_n,\Sigma_n)\to (X,\mu,\Sigma)$ which is an asymptotic extension. In particular, $\mathcal{R}_{\Sigma}$ is $\mu$-sofic.
\end{thm}

To verify the treeability of $\mathcal{R}_{\Sigma_n}$ in this theorem, we will use a variant of Theorem \ref{thmJ1}:

\begin{thm}[Theorem \ref{fact2}] \label{fact2'}
Let $(X,\Sigma)$ be a bounded-degree planar Borel complex without singular vertices.
If $\mathrm{H}_1(\Sigma,\mathbb{Z}_2)$ is boundedly generated, then $\mathcal{R}_\Sigma$ is measure-treeable.
\end{thm}

Therefore, a key ingredient in the proof of Theorem \ref{thm2'} is the construction of planar coverings satisfying certain properties. We utilize a family of \textit{tracks} on planar complexes (see Section \ref{sec4}), a concept originally introduced for $2$-dimensional complexes by Dunwoody \cite{D2}. Note that a (closed) track on a planar complex is identified with a simple closed curve on the topological realization of the complex. We introduce a notion called irreducibility for these tracks inspired by \cite{C,D1}. The following observation states that the family of irreducible tracks has desirable properties:

\begin{prop} \label{prop0}
Let $\Sigma$ be a connected locally-finite planar complex without singular vertices. Then the set of irreducible tracks on $\Sigma$ is pairwise-nested and generates the fundamental group of $\Sigma$.
\end{prop}

Now let $\Sigma$ be a connected bounded-degree planar complex without singular vertices. For a positive integer $n$, we can construct a normal covering $\Sigma_n$ of $\Sigma$ so that $\pi_1(\Sigma_n)$ corresponds to the normal subgroup of $\pi_1(\Sigma)$ generated by the set of tracks on $\Sigma$ of length at most $n$. Then $\Sigma_n$ is planar and $\mathrm{H}_1(\Sigma_n,\mathbb{Z}_2)$ is boundedly generated. Moreover, $\Sigma_n$ ``converges" to $\Sigma$ since the normal subgroup corresponding to $\pi_1(\Sigma_n)$ increases to $\pi_1(\Sigma)$ as $n\to \infty$. Since this construction is canonical, we can apply it to pmp complexes and prove Theorem \ref{thm2'}.

\subsection{Related questions}

Sol\'{e}-Pi proves that every unimodular random one-ended minor-excluded graph is sofic \cite{S}. Here a graph $G$ is said to be \textit{minor-excluded} if there is a finite graph $H$ such that $G$ does not contain $H$ as a minor. This result widely extends Tim\'{a}r's work \cite{T}. The proof in \cite{S} shows that these graphs have a property close to treeability, although treeability is not proved.
It is natural to expect the soficity of unimodular random minor-excluded graphs without the assumption on the number of ends \cite[Conjecture 1.6]{S}, which would extend our Theorem \ref{thm1}. However, applying our method to these graphs requires many modifications since we heavily rely on the planarity.

Also, robustness of Theorems \ref{thm1} and \ref{thm3} under quasi-isometry can be a subject of interest. For instance, is every unimodular random graph which is  almost surely quasi-isometric to a planar graph sofic? Indeed, finitely generated groups having quasi-planar Cayley graphs are completely classified \cite{M}, and they are known to be treeable \cite{Ga}. Additionally, in the case of quasi-trees, it has been proved that any Borel graph whose components are quasi-trees is (Borel-)treeable \cite{CPTT}. However, the above question is open even in the one-ended case.

\subsection{Organization of the paper}
In Section \ref{sec2}, we recall basic notions of measure-preserving equivalence relations and Borel complexes. In Section \ref{sec3}, tree decompositions of graphs and irreducibility of cuts are explained. In Section \ref{sec4}, we investigate tracks on planar complexes and coverings of these complexes associated with families of tracks. Finally in Section \ref{sec5}, we prove the main theorems.

\section*{Notations and terminology}
Throughout this paper, a graph means a simplicial graph (i.e., a $1$-dimensional simplicial complex).

\begin{nta}\label{nta1}
Let $G$ be a graph.
\begin{enumerate}
    \item $VG$ and $EG$ are the sets of vertices and edges of $G$, respectively. 
    \item $\overrightarrow{E}G=\{(x,y)\in (VG)^2\mid \{x,y\}\in EG\}$ is the set of oriented edges of $G$. For $e=(x,y)\in \overrightarrow{E}G$, set $s(e)=x,t(e)=y$ and $e^{-1}=(y,x)\in \overrightarrow{E}G$.
    \item For $A\subset VG$, set
    \begin{align*}
        &\partial_\mathrm{iv}^G(A)=\{x\in A\mid \exists y\in VG\setminus A,\ \{x,y\}\in EG\},\\
        &\partial_\mathrm{ov}^G(A)= \partial_\mathrm{iv}^G(VG\setminus A)\ \textup{ and}\\
        &\partial_\mathrm{e}^G(A)=\{e\in EG\mid e\cap A\neq \varnothing,e\cap(VG\setminus A)\neq \varnothing\}.
    \end{align*}
    \item For $x\in VG$, set $\mathrm{N}_G(x)=\{y\in VG\mid \{x,y\}\in EG\}$.
\end{enumerate}
Let $\Sigma$ be a simplicial complex. 
\begin{enumerate}
    \item $V\Sigma$, $E\Sigma$ and $F\Sigma$ are the sets of vertices ($0$-simplices), edges ($1$-simplices) and faces ($2$-simplices) of $\Sigma$, respectively.
    \item $\Sigma^1=V\Sigma\cup E\Sigma$ is the $1$-skeleton of $\Sigma$. 
    \item For a subset $A\subset V\Sigma$, let $\Sigma[A]=\{\sigma\in\Sigma\mid \sigma\subset A\}$ be the induced complex on $A$.
\end{enumerate}
\end{nta}

\begin{dfn}
Let $G$ be a graph.
\begin{enumerate}
    \item A \textit{path} on $G$ (or a $G$-\textit{path}) is a sequence of oriented edges $e_1,e_2,...,e_n\in \overrightarrow{E}G$ such that $s(e_i)=t(e_{i+1})$ for $i=1,2,...,n-1$. Then this path $p$ is denoted by $p=e_1e_2\cdots e_n$, and is called a path from $s(p)=s(e_1)$ to $t(p)=t(e_n)$. The length of $p$ is $n$ and denoted by $\mathrm{len}_G(p)$.
    \item If $G$ is connected, let $d_G:(VG)^2\to \{0,1,2,...\}$ be the path metric, i.e., $d_G(x,y)$ is the length of the shortest paths on $G$ from $x$ to $y$.
    \item A path $p=e_1e_2\cdots e_n$ is said to be simple if $s(e_1),s(e_2),...,s(e_n)$ are distinct from each other and $t(e_1),t(e_2),...,t(e_n)$ are distinct from each other.
    \item A path $p$ is said to be closed if $s(p)=t(p)$.
    \item The homology group $\mathrm{H}_1(G,\mathbb{Z}_2)$ is called the cycle space of $G$ and identified with a subspace of $\mathbb{Z}_2^{\oplus{EG}}$. The \textit{size} of an element $c\in\mathrm{H}_1(G,\mathbb{Z}_2)$ is the number of $e\in EG$ such that $c(e)=1$ as an element of $\mathbb{Z}_2^{\oplus{EG}}$, which is denoted by $|c|$.
    \item For a closed path $p=e_1e_2\cdots e_n$, the element $[p]=\delta_{e_1}+\delta_{e_2}+\cdots+\delta_{e_n}\in\mathrm{H}_1(G,\mathbb{Z}_2)$ is said to be represented by $p$. An element of $\mathrm{H}_1(G,\mathbb{Z}_2)$ represented by a simple closed path is called a simple cycle of $G$. 
\end{enumerate}
The terminology from (i) to (iv) is also applied for a simplicial complex $\Sigma$ in the sense that $G=\Sigma^1$.
\end{dfn}

\section{Preliminaries}\label{sec2}

\subsection{Measured equivalence relations}

Let $(X,\mu)$ be a standard ($\sigma$-finite) measure space with $\mu(X)>0$.

\begin{dfn}
A \textit{countable Borel equivalence relation} on $X$ is a Borel subset $\mathcal{R}\subset X\times X$ such that $\mathcal{R}$ is an equivalence relation on $X$ and every equivalence class is at most countable.

Then, a \textit{partial inner automorphism} of $\mathcal{R}$ is a Borel isomorphism $\varphi:\mathrm{dom}\varphi\to\mathrm{img}\varphi$ between Borel subsets of $X$ such that $\{(x,\varphi x)\mid x\in\mathrm{dom}\varphi\}\subset \mathcal{R}$.

We say that $\mathcal{R}$ is a \textit{measure-preserving equivalence relation} on $(X,\mu)$ if for every partial inner automorphism $\varphi$ of $\mathcal{R}$, we have $\mu(\mathrm{dom}\varphi)=\mu(\mathrm{img}\varphi)$. If $\mu(X)=1$, then we also say that $\mathcal{R}$ is \textit{probability-measure-preserving} (or \textit{pmp}, for short).
\end{dfn}

\begin{nta} Let $\mathcal{R}$ be a measure-preserving equivalence relation on $(X,\mu)$.
\begin{enumerate}
    \item Let $[[\mathcal{R}]]$ be the semigroup of partial inner automorphisms of $\mathcal{R}$ with the canonical semigroup operation. We identify two elements $\varphi,\varphi'\in[[\mathcal{R}]]$ if $\mu(\mathrm{dom}\varphi\bigtriangleup\mathrm{dom}\varphi')=0$ and $\varphi x=\varphi' x$ for almost every $x\in\mathrm{dom}\varphi$.
    \item Define the trace $\tau_\mu:[[\mathcal{R}]]\to [0,\infty]$ by $\tau_\mu(\varphi)=\mu(\{x\in\mathrm{dom}\varphi\mid x=\varphi x\})$.
    \item Define the (extended) metric $d_\mu:[[\mathcal{R}]]\times[[\mathcal{R}]]\to [0,\infty]$ by
    \begin{align*}
        d_\mu(\varphi,\psi)
        &=\mu((\mathrm{dom}\varphi\cup\mathrm{dom}\psi)\setminus\{x\in\mathrm{dom}\varphi\cap\mathrm{dom}\psi\mid \varphi x=\psi x\}),
    \end{align*}
    which is equal to $\mu(\mathrm{dom}\varphi)+\mu(\mathrm{dom}\psi)-\tau_\mu(\psi^{-1}\varphi)$ in the case of $\mu(X)<\infty$.
    \item For $x\in X$, let $[x]_\mathcal{R}=\{y\in X\mid (x,y)\in\mathcal{R}\}$ be the equivalence class containing $x$.
    \item For a Borel subset $A\subset X$ of positive measure, let $\mathcal{R}|_A=\mathcal{\mathcal{R}}\cap A^2$ be the measure-preserving equivalence relation on $(A,\mu|_A)$.
\end{enumerate}
\end{nta}

We define soficity of measure-preserving equivalence relations. For pmp equivalence relations, it is introduced by \cite{EL}, and we follow the formulation of \cite[Definition 3.33]{CGS}.

\begin{dfn}
Let $\mathcal{R}$ be a measure-preserving equivalence relation on $(X,\mu)$.
\begin{enumerate}
    \item Suppose $\mu(X)=1$. Then $\mathcal{R}$ is $\mu$-\textit{sofic} if for every finite subset $K\subset[[\mathcal{R}]]$ and every $\varepsilon>0$, there exists a finite pmp equivalence relation $(Y,\nu,\mathcal{S})$ (i.e., for almost every $y\in Y$, the set $[y]_\mathcal{S}$ is finite) and a map $\rho:K\to [[\mathcal{S}]]$ such that if $\varphi,\psi,\varphi\psi\in K$, then
    \begin{equation*}
        d_\nu(\rho(\varphi),\rho(\psi)) <\varepsilon \textup{ and }|\tau_\mu(\varphi)-\tau_\nu(\rho(\varphi))|<\varepsilon.
    \end{equation*}
    \item Suppose $0<\mu(X)<\infty$. Then $\mathcal{R}$ is $\mu$-\textit{sofic} if the pmp equivalence relation $\mathcal{R}$ on $(X,\mu(X)^{-1}\mu)$ is $\mu(X)^{-1}\mu$-\textit{sofic}.
    \item Suppose $\mu(X)=\infty$. Then $\mathcal{R}$ is $\mu$-\textit{sofic} if for every $A\subset X$ with $0<\mu(A)<\infty$, the equivalence relation $(A,\mu|_A,\mathcal{R}|_A)$ is $\mu|_A$-sofic.
\end{enumerate} 
\end{dfn}

\begin{dfn}
Let $(X,\mu,\mathcal{R})$ and $(Y,\nu,\mathcal{S})$ be measure-preserving equivalence relations.
\begin{enumerate}
    \item A \textit{factor map} $\pi:(X,\mu,\mathcal{R})\to(Y,\nu,\mathcal{S})$ is a measure-preserving map $\pi:(X,\mu)\to (Y,\nu)$ such that for almost every $x\in X$, we have $\pi([x]_\mathcal{R})=[\pi(x)]_\mathcal{S}$.
    \item A factor map $\pi:(X,\mu,\mathcal{R})\to(Y,\nu,\mathcal{S})$ is called an \textit{extension} if for almost every $x\in X$, the restriction $\pi|_{[x]_\mathcal{R}}:[x]_\mathcal{R}\to[\pi(x)]_\mathcal{S}$ is bijective. If moreover $\pi:(X,\mu)\to (Y,\nu)$ is a measure-isomorphism, then $\pi:(X,\mu,\mathcal{R})\to(Y,\nu,\mathcal{S})$ is called an isomorphism.
    \item An \textit{embedding} $\iota:(X,\mu,\mathcal{R})\to(Y,\nu,\mathcal{S})$ is an injective Borel map $\iota:X\to Y$ such that $\iota:(X,\mu,\mathcal{R})\to (\iota(X),\nu|_{\iota(X)},\mathcal{S}|_{\iota(X)})$ is an isomorphism.
\end{enumerate}
\end{dfn}

\begin{rmk}\label{rmk1}
Let $\iota:(X,\mathcal{R})\to(Y,\mathcal{S})$ be a \textit{Borel embedding} of countable Borel equivalence relations, i.e., an injective Borel map $\iota:X\to Y$ such that for every $x\in X$, we have $\iota([x]_\mathcal{R})=[\iota(x)]_\mathcal{S}\cap \iota(X)$. If $(X,\mu,\mathcal{R})$ is measure-preserving, then there exists a unique measure $\nu$ on $Y$ supported on $[\iota(X)]_\mathcal{S}=\bigcup_{x\in X}[\iota(x)]_\mathcal{S}$ such that $(Y,\nu,\mathcal{S})$ is measure-preserving and $\iota:(X,\mu,\mathcal{R})\to(Y,\nu,\mathcal{S})$ is an embedding.
\end{rmk}

The following will be used many times:

\begin{lem}
Let $(X,\mu,\mathcal{R})$ and $(Y,\nu,\mathcal{S})$ be measure-preserving equivalence relations. Suppose that $\mathcal{S}$ is $\nu$-sofic.
\begin{enumerate}
    \item If $\pi:(Y,\nu,\mathcal{S})\to(X,\mu,\mathcal{R})$ is an extension, then $\mathcal{R}$ is $\mu$-sofic.
    \item If $\iota:(X,\mu,\mathcal{R})\to(Y,\nu,\mathcal{S})$ is an embedding, then $\mathcal{R}$ is $\mu$-sofic.
\end{enumerate}
\end{lem}

\begin{proof}
If $\pi:(Y,\nu,\mathcal{S})\to(X,\mu,\mathcal{R})$ is an extension, then for every $\varphi\in X$, there exists a unique element $\pi^*(\varphi)\in[[\mathcal{S}]]$ such that $\mathrm{dom}(\pi^*(\varphi))=\pi^{-1}(\mathrm{dom}\varphi)$ and for almost every $y\in \mathrm{dom}(\pi^*(\varphi))$, we have $\pi(\pi^*(\varphi)y)=\varphi(\pi(y))$. Then the map $\pi^*:[[\mathcal{R}]]\to[[\mathcal{S}]]$ is a homomorphism of semigroups and preserves the traces. Hence assertion (i) follows.

Assertion (ii) is clear from the definition of soficity.
\end{proof}

\begin{dfn}
Let $\{\mathcal{R}_i\}_{i\in I}$ be a countable family of countable Borel equivalence relations on $X$. 
We say that the family $\{\mathcal{R}_i\}_{i\in I}$ is \textit{freely intersecting} if the following holds: If a finite sequence $\{x_1,x_2,...,x_n\}\subset X$ satisfies $x_k\neq x_{k+1}$ and $(x_k,x_{k+1})\in \mathcal{R}_{i(k)}$ with $i(k)\in I$ for every $1\leq k\leq n-1$, then we have $x_1\neq x_n$ or $i(k)=i(k+1)$ for some $1\leq k\leq n-1$.

If $\{\mathcal{R}_i\}_{i\in I}$ is freely intersecting, then the minimal equivalence relation on $X$ containing all $\mathcal{R}_i$ is denoted by $\bigast_{i\in I}\mathcal{R}_i$ and called the \textit{free product} of $\{\mathcal{R}_i\}_{i\in I}$.
\end{dfn}

\begin{dfn}
Let $(X,\mathcal{R})$ be a countable Borel equivalence relation.
\begin{enumerate}
    \item The relation $\mathcal{R}$ is \textit{treeable} if there exists an acyclic Borel graph $T$ on $X$ such that $\mathcal{R}$ is the connected relation of $T$.
    \item For a standard measure $\mu$ on $X$, the relation $\mathcal{R}$ is said to be $\mu$-\textit{treeable} if there exists a $\mu$-conull subset $X'\subset X$ such that $(X',\mathcal{R}|_{X'})$ is treeable.
    \item The relation $\mathcal{R}$ is \textit{measure-treeable} if $\mathcal{R}$ is $\mu$-treeable for any standard measure $\mu$ on $X$.
\end{enumerate}
\end{dfn}

\begin{rmk}\label{lem13}
A countable Borel equivalence relation is treeable if and only if it is a free product of finite equivalence relations. In particular, treeability is preserved by taking free products.
\end{rmk}

\begin{prop}\cite[Theorem 4]{EL} \label{prop2}
Every pmp $\mu$-treeable equivalence relation on $(X,\mu)$ is $\mu$-sofic.
\end{prop}

Moreover, the following proposition is well known to experts. For example, it easily follows from \cite[Lemma 7.2]{DKP1}.

\begin{prop}\label{lem14}
Let $\mathcal{R}=\bigast_{i\in I}\mathcal{R}_i$ be a measure-preserving equivalence relation on $(X,\mu)$.
If $\mathcal{R}_i$ is $\mu$-sofic for every $i\in I$, then so is $\mathcal{R}$.
\end{prop}

\subsection{Borel complexes}

We recall the definition of Borel complexes.

\begin{dfn}
A \textit{Borel complex} on $X$ is a locally finite simplicial complex such that the set of vertices is $X$ and the set of simplices is a Borel subset of the set of finite subsets of $X$. Then we also say that $(X,\Sigma)$ is a Borel complex. A $1$-dimensional Borel complex is called a \textit{Borel graph}.
\end{dfn}

Throughout this article, we only consider Borel complexes of dimension at most $2$.

\begin{nta}\label{nta2}
Let $\Sigma$ be a Borel complex on $X$.
\begin{enumerate}
    \item Let $\mathcal{R}_{\Sigma}=\mathcal{R}_{\Sigma^1}$ denote the connected relation of $\Sigma$ on $X$.
    \item For $x\in X$, let $\Sigma_x$ denote the connected component of $\Sigma$ containing $x$.
\end{enumerate}
\end{nta}

\begin{dfn}
A measure-preserving complex on $(X,\mu)$ is a Borel complex on $X$ such that $\mathcal{R}_{\Sigma}$ is pmp on $(X,\mu)$. Then we also say that $(X,\mu,\Sigma)$ is a \textit{measure-preserving complex}. If $\mu(X)=1$, then a measure-preserving complex is called a pmp complex. 

For a measure-preserving complex $(X,\mu,\Sigma)$, let $[[\Sigma]]$ denote the set of $\varphi\in[[\mathcal{R}_\Sigma]]$ such that for almost every $x\in\mathrm{dom}\varphi$, we have $\{x,\varphi x\}\in E\Sigma$.
\end{dfn}

Now we define covering factor maps between measure-preserving complexes.

\begin{dfn}\label{dfn1}
Let $(X,\mu,\Sigma),(Y,\nu,\Sigma')$ and $(X_n,{\mu_n},\Sigma_n)\ (n=1,2,...)$ be measure-preserving complexes. 
\begin{enumerate}
    \item A \textit{covering factor map} $\pi:(X,\mu,\Sigma)\to (Y,\nu,\Sigma')$ is a factor map $\pi:(X,\mu,\mathcal{R}_\Sigma)\to (Y,\nu,\mathcal{R}_{\Sigma'})$ such that for almost every $x\in X$, the restriction $\pi|_{[x]_\mathcal{R}}$ gives a covering map $\Sigma_x\to \Sigma'_{\pi(x)}$ of simplicial complexes.
    \item If $\pi:(X,\mu,\Sigma)\to (Y,\nu,\Sigma')$ is a covering map, then for every $\psi\in [[\Sigma']]$, the element $\varphi\in[[\Sigma]]$ is defined so that $\psi(\pi(x))=\pi(\varphi x)$ for almost every $x\in \mathrm{dom}\varphi=\pi^{-1}(\mathrm{dom}\psi)$. We call $\varphi$ the \textit{lift} of $\psi$ to with respect to $\pi$.
    \item Let $\pi_n:(X_n,{\mu_n},\Sigma_n)\to (Y,\nu,\Sigma')$ be a sequence of covering factor maps. This sequence is called an \textit{asymptotic extension} if for all finitely many $\psi^1,\psi^2,...,\psi^{k}\in [[\Sigma']]$, the lifts $\varphi_n^i\in [[\Sigma_n]]$ of $\psi^i$ with respect to $\pi_n$ ($i=1,2,...,k$) satisfy 
\begin{equation*}
    \tau_{\mu_n}(\varphi_n^1\varphi_n^2\cdots\varphi_n^{k})\to
    \tau_\nu(\psi^1\psi^2\cdots\psi^{k})
    \quad \textup{as }n\to \infty.
\end{equation*}
\end{enumerate}
\end{dfn}

\begin{lem}\label{lem1}
Let $\pi_n:(X_n,{\mu_n},\Sigma_n)\to (X,\mu,\Sigma)$ be a sequence of covering factor maps between pmp complexes. Suppose that the sequence $\pi_n$ is an asymptotic extension, and $\mathcal{R}_{\Sigma_n}$ is sofic for every $n$. Then $\mathcal{R}_{\Sigma}$ is sofic.
\end{lem}

\begin{proof}
Let $K\subset[[\mathcal{R}_\Sigma]]$ be a finite set and $\varepsilon>0$. Since $\mathcal{R}_{\Sigma_n}$ is sofic for every $n$, it suffices to find a positive integer $n$ and a map $\rho:K\to[[\mathcal{R}_{\Sigma_n}]]$ such that if $\varphi,\psi,\varphi\psi\in K$, then we have $d_{\mu_n}(\rho(\varphi),\rho(\psi)) <\varepsilon$ and $|\tau_\mu(\varphi)-\tau_{\mu_n}(\rho(\varphi))|<\varepsilon$.

Take a positive integer $k$ so that for every $\varphi\in K$, there exist a finite set $J_\varphi$ and a family $\{\varphi^{1,j},...,\varphi^{k,j}\}_{j\in J_\varphi}\subset [[\Sigma]]$ such that $A_{\varphi,j}=\mathrm{dom}(\varphi^{1,j}\varphi^{2,j}\cdots\varphi^{k,j})\ (j\in J_\varphi)$ are pairwise disjoint subsets of $\mathrm{dom}\varphi$, and we have
\begin{align}
    &\varphi|_{A_{\varphi,j}}=\varphi^{1,j}\varphi^{2,j}\cdots\varphi^{k,j}
     \quad\textup{and} \label{eq4}
     \\
    &\mu\left(\mathrm{dom}\varphi\setminus\bigsqcup_{j\in J_\varphi}A_{\varphi,j}\right) <\frac{\varepsilon}{12}. \label{eq5}
\end{align}
Set $A_\varphi=\bigsqcup_{j\in J_\varphi}A_{\varphi,j}$ for each $\varphi\in K$. Note that for all $\varphi,\psi,\theta \in K$, we have by \eqref{eq5}
\begin{equation}
    d_\mu(\theta^{-1}\varphi\psi, (\theta|_{A_\theta})^{-1}\varphi|_{A_\varphi}\psi|_{A_\psi}) <\frac{\varepsilon}{4}. \label{eq2}
\end{equation}
For each $n\in\mathbb{Z}_{\geq 1}$, $\varphi\in K$, $i\in \{1,...,k\}$ and $j\in J_\varphi$, let $\varphi_n^{i,j}$ be the lift of $\varphi^{i,j}$ with respect to $\pi_n$, and set $\varphi_n^j=\varphi_n^{1,j}\varphi_n^{2,j}\cdots\varphi_n^{k,j}$.

Now since $\pi_n$ is an asymptotic extension, we can take a positive integer $n$ so that for all $\varphi,\psi,\theta \in K$, we have
\begin{align}
    &\sum_{j\in J_\varphi}
    \left|\tau_{\mu_n}(\varphi_n^j) -
    \tau_\mu (\varphi|_{A_{\varphi,j}})\right| <\frac{\varepsilon}{2} \quad \textup{and} \label{eq3}
    \\
    &\sum_{j\in J_\varphi}\sum_{j'\in J_\psi}\sum_{j''\in J_\theta}
    \left|\tau_{\mu_n}\left((\theta_n^{j''})^{-1}\varphi_n^j\psi_n^{j'}\right) -
    \tau_\mu \left((\theta^{j''}|_{A_{\theta,j''}})^{-1}\varphi|_{A_{\varphi,j}}\psi^{j'}|_{A_{\psi,j'}}\right)\right| <\frac{\varepsilon}{4}. \label{eq1}
\end{align}
For $\varphi\in K$, define $\rho(\varphi)\in [[\mathcal{R}_{\Sigma_n}]]$ by $\mathrm{dom}\rho(\varphi)=\pi_n^{-1}(\bigsqcup_{j\in J_\varphi}A_{\varphi,j})$ and $\rho(\varphi)x=\varphi_n^{1,j}\cdots\varphi_n^{k,j}x$
for every $x\in \pi_n^{-1}(A_{\varphi,j})$. We verify that $\rho(\varphi)$ is indeed injective. By construction, we have $\pi_n\circ\rho(\varphi)=\varphi\circ\pi_n$ on $\mathrm{dom}\rho(\varphi)$. Thus if $x,y\in \mathrm{dom}\rho(\varphi)$ satisfy $\pi_n(x)\neq \pi_n(y)$, then it is clear that $\rho(\varphi)x\neq \rho(\varphi)y$. On the other hand, if distinct $x,y\in \mathrm{dom}\rho(\varphi)$ satisfy $\pi_n(x)=\pi_n(y)\in A_{\varphi,j}$ for some $j$, then we have $\rho(\varphi)x=\varphi_n^{1,j}\cdots\varphi_n^{k,j}x$ and $\rho(\varphi)y=\varphi_n^{1,j}\cdots\varphi_n^{k,j}y$, which are distinct. Hence we have $\rho(\varphi)\in [[\mathcal{R}_{\Sigma_n}]]$.

Let $\varphi,\psi\in K$ and $\theta=\varphi\psi\in K$. By \eqref{eq1}, we have 
\begin{equation*}
    \left|\tau_{\mu_n}(\rho(\theta)^{-1}\rho(\varphi)\rho(\psi))-\tau_\mu((\theta|_{A_\theta})^{-1}\varphi|_{A_\varphi}\psi|_{A_\psi})\right| <\frac{\varepsilon}{4},
\end{equation*}
and then by \eqref{eq2}
\begin{equation*}
     \left|\tau_{\mu_n}(\rho(\varphi\psi)^{-1}\rho(\varphi)\rho(\psi))-\mu(\mathrm{dom}(\varphi\psi))\right| =\left|\tau_{\mu_n}(\rho(\theta)^{-1}\rho(\varphi)\rho(\psi))-\tau_\mu(\theta^{-1}\varphi\psi)\right| <\frac{\varepsilon}{2},
\end{equation*}
where we use $\theta=\varphi\psi$ and $\tau_\mu(\theta^{-1}\varphi\psi)=\mu(\mathrm{dom}(\varphi\psi))$.
Since $\mathrm{dom}\rho(\varphi\psi),\mathrm{dom}(\rho(\varphi)\rho(\psi))\subset \pi^{-1}(\mathrm{dom}(\varphi\psi))$, we have
\begin{align*}
    d_{\mu_n}(\rho(\varphi)\rho(\psi),\rho(\varphi\psi))
    &=\mu_n(\mathrm{dom}(\rho(\varphi)\rho(\psi)))+\mu_n(\mathrm{dom}\rho(\varphi\psi))-2\tau_{\mu_n}(\rho(\theta)^{-1}\rho(\varphi)\rho(\psi))\\
    &\leq 2\mu(\mathrm{dom}(\varphi\psi)) - 2\tau_{\mu_n}(\rho(\theta)^{-1}\rho(\varphi)\rho(\psi))<\varepsilon.
\end{align*}
Also, we have by inequality \eqref{eq3},
\begin{equation*}
    |\tau_{\mu_n}(\rho(\varphi)) - \tau_{\mu}(\varphi)|
    \leq |\tau_{\mu_n}(\rho(\varphi)) - \tau_{\mu}(\varphi|_{A_\varphi})|+ d_\mu(\varphi|_{A_\varphi},\varphi)<\frac{\varepsilon}{2}+\frac{\varepsilon}{12}<\varepsilon.
\end{equation*}
\end{proof}

Finally, we introduce the barycentric subdivisions of Borel complexes.

\begin{dfn}
Let $\Sigma$ be a simplicial complex. The barycentric subdivision of $\Sigma$ is the simplicial complex $\Sigma'$ such that $V\Sigma'=\Sigma$ and the set of $n$-simplices of $\Sigma'$ is
\begin{equation*}
    \{\{\sigma_1,...,\sigma_n\}\subset \Sigma\mid \sigma_1\subsetneq\sigma_2\subsetneq\cdots\subsetneq\sigma_n\}.
\end{equation*}
\end{dfn}

For a Borel complex $(X,\Sigma)$, the barycentric subdivision $(X',\Sigma')$, where $X'=V\Sigma'=\Sigma$, is also a Borel complex. Moreover, the inclusion map $X\hookrightarrow X'$ induces a Borel embedding $(X,\mathcal{R}_{\Sigma})\to (X',\mathcal{R}_{\Sigma'})$.

\subsection{Measured groupoids}

We also need some preparation on pmp groupoids.

\begin{dfn}
A \textit{countable Borel groupoid} $\mathcal{G}$ on $X$ is a standard Borel space $\mathcal{G}$ equipped with a groupoid structure such that the set of objects $\mathcal{G}^0$ is a Borel subset of $\mathcal{G}$ and identified with $X$, and the source and target maps $s,t:\mathcal{G}\to \mathcal{G}^0=X$ are countable-to-one Borel maps.

For a countable Borel groupoid $\mathcal{G}$ on $X$, the set $\mathcal{R}_\mathcal{G}=\{(s(g),t(g))\in X^2\mid g\in\mathcal{G} \}$ is a countable Borel equivalence relation on $X$. If $\mathcal{R}_\mathcal{G}$ is a pmp equivalence relation on $(X,\mu)$, then $\mathcal{G}$ is called a \textit{pmp groupoid} on $(X,\mu)$. 
\end{dfn}

\begin{exm}[{Principal groupoids}]
A pmp equivalence relation $\mathcal{R}$ on $(X,\mu)$ is a pmp groupoid on $(X,\mu)$ by setting $s(x,y)=x$ and $t(x,y)=y$ for $(x,y)\in\mathcal{R}$.
\end{exm}

\begin{dfn}
Let $(X,\mu,\mathcal{G})$ and $(Y,\nu,\mathcal{H})$ be pmp groupoids. A homomorphism $\pi:\mathcal{G}\to\mathcal{H}$ is called an \textit{extension} if $\pi|_X:(X,\mu)\to (Y,\nu)$ is measure-preserving, and for almost every $x\in X$, the map $\pi|_{s^{-1}(x)}:s^{-1}(x)\to s^{-1}(\pi(x))$ is bijective. If moreover $\mathcal{G}$ is a pmp equivalence relation on $(X,\mu)$, then $\pi$ is called a \textit{principal extension}.
\end{dfn}

\begin{rmk}
Every pmp groupoid $\mathcal{G}$ admits a principal extension $\mathcal{R}\to\mathcal{G}$. For instance, it is given by the \textit{Bernoulli action} of $\mathcal{G}$ (see \cite[Section 3]{BT}). 
\end{rmk}

The next example is one of the keys to prove Theorem \ref{thm2'}.

\begin{exm}[{Fundamental groupoids}]\label{exm2}
Let $(X,\mu,\Sigma)$ be a pmp $2$-dimensional simplicial complex. Let $\mathrm{Path}(\Sigma)$ be the set of paths on $\Sigma$. Then $\mathrm{Path}(\Sigma)$ is a pmp groupoid on $X$, where the product and inverse of paths are defined by concatenation and orientation-reversing, respectively.

Let $\sim$ be the homotopy equivalence relation on $\Sigma$, i.e., for $p,q\in\mathrm{Path}(\Sigma)$, we have $p\sim q$ if $(s(p),t(p))=(t(q),s(q))$ and the closed path $pq^{-1}$ represents the trivial element of the fundamental group $\pi_1(\Sigma,s(p))$. Then the fundamental groupoid $\Pi_1(\Sigma)$ on $X$ is defined by $\Pi_1(\Sigma)=\mathrm{Path}(\Sigma)/\sim$. 

Since $\sim$ is a smooth equivalence relation on $\mathrm{Path}(\Sigma)$, the set $\Pi_1(\Sigma)$ is regarded as a standard Borel space. Then $\Pi_1(\Sigma)$ is a pmp groupoid on $(X,\mu)$ with the groupoid structure induced from $\mathrm{Path}(\Sigma)$. This is called the \textit{fundamental groupoid} of $\Sigma$.

The projection $\varepsilon: \Pi_1(\Sigma)\to \mathcal{R}_{\Sigma},\ g\mapsto (s(g),t(g))$ is a factor map. Take a principal extension $\rho:(Y,\nu,\widetilde{\mathcal{R}})\to (X,\mu,\Pi_1(\Sigma))$ and set $\pi=\varepsilon\circ\rho:(Y,\nu,\widetilde{\mathcal{R}})\to (X,\mu,\mathcal{R}_{\Sigma})$. We regard each $e=(x,x')\in \overrightarrow{E}\Sigma$ as a path from $x$ to $x'$, and let $\widetilde{e}\in \Pi_1(\Sigma)$ be the element represented by this path. Let $\widetilde{\Sigma}^1$ be the Borel graph on $Y$ defined by
\begin{equation*}
    \overrightarrow{E}\widetilde{\Sigma}^1=\{(y,y')\in\widetilde{\mathcal{R}}\mid \rho(y,y')=\widetilde{e},\ \textup{where }e=(\pi(y),\pi(y'))\in \overrightarrow{E}\Sigma\}.
\end{equation*}
Now let $\widetilde{\Sigma}$ be the $2$-dimensional Borel simplicial complex on $Y$ such that the $1$-skeleton of $\widetilde{\Sigma}$ is $\widetilde{\Sigma}^1$ and
\begin{equation*}
    F\widetilde{\Sigma}=\{\{y_1,y_2,y_3\}\subset Y\mid \{y_i,y_j\}\in E\widetilde{\Sigma} \textup{ if }i\neq j,\ \pi(\{y_1,y_2,y_3\})\in F\Sigma\}.
\end{equation*}
Then $\pi:(Y,\nu,\widehat{\Sigma})\to (X,\mu,\Sigma)$ is a covering factor map, and for almost every $y\in Y$, the map $\pi|_{\widetilde{\Sigma}_y}:\widetilde{\Sigma}_y\to \Sigma_{\pi(y)}$ is the universal covering map.

\end{exm}

\section{Tree decompositions}\label{sec3}

\subsection{Basic theory}\label{sec3.1}
Following \cite{J}, we review tree decompositions of graphs and their applications to Borel graphs.

First, let $G$ be a connected graph.

\begin{dfn}
A \textit{tree decomposition} of $G$ is a pair a tree $T$ and a family $\{V_t\}_{t\in VT}$ of subsets of $VG$ such that
\begin{enumerate}
    \item $VG=\bigcup_{t\in VT}V_t$,
    \item for every $\{u,v\}\in EG$, there exists $t\in VT$ such that $\{u,v\}\subset V_t$, and
    \item if $t_1,t_2,t_3\in VT$ satisfy $d_T(t_1,t_2)=d_T(t_1,t_3)+d_T(t_3,t_2)$, then $V_{t_1}\cap V_{t_2}\subset V_{t_3}$.
\end{enumerate}
For $\{t_1,t_2\}\in ET$, the set $V_{t_1}\cap V_{t_2}$ is called an \textit{adhesion} of the tree decomposition $(T,\{V_t\}_{t\in VT})$. This tree decomposition is said to be \textit{proper} if $\{V_t\}_{t\in VT}$ is pairwise distinct.

A tree decomposition $(T,\{V_t\}_{t\in VT})$ is sometimes referred as just $T$ if $\{V\}_{t\in VT}$ is already given.
\end{dfn}

\begin{dfn}
A \textit{separation} of $G$ is a pair $\{A_1,A_2\}$ of non-empty proper subsets of $VG$ such that $VG=A_1\cup A_2$, $|A_1\cap A_2|<\infty$ and $\partial_\mathrm{iv}^GA_1,\partial_\mathrm{iv}^GA_2 \subset A_1\cap A_2$. If $|A_1\cap A_2|=n$, then $(A_1,A_2)$ is called an $n$-\textit{separation}. An \textit{ordered separation} of $G$ is a ordered pair $(A_1,A_2)$ such that $\{A_1,A_2\}$ is a separation of $G$. 

Let $\mathrm{Sep}(G)$ ($\mathrm{Sep}^*(G)$, resp.) denote the sets of separations (ordered separations, resp.) of $G$. The order $\leq$ on $\mathrm{Sep}^*(G)$ is defined by $(A_1,A_2)\leq (B_1,B_2)$ if and only if $A_1\subset B_1$ and $B_2\subset A_2$. Two separations $\{A_1,A_2\}$ and $\{B_1,B_2\}$ are said to be \textit{nested} if $(A_i,A_j)\leq (B_k,B_l)$ for some $\{i,j\}=\{k,l\}=\{1,2\}$.

A \textit{separation system} of $G$ is a pairwise-nested family $\mathcal{S}\subset \mathrm{Sep}(G)$ such that
if two elements $\{A_1,A_2\}$ and $\{B_1,B_2\}$ of $\mathcal{S}$ satisfies $(A_1,A_2)\lneq (B_1,B_2)$, then there exist finitely many $\{C_1,C_2\}\in\mathcal{S}$ such that $(A_1,A_2)\lneq (C_1,C_2)\lneq (B_1,B_2)$.
\end{dfn}

\begin{lem}[{\cite[Lemma 2.3]{CHM}}] \label{lem20}
Let $\mathcal{S}\subset\mathrm{Sep}(G)$ be a pairwise-nested family. If $\mathcal{S}$ satisfies $\sup_{\{A_1,A_2\}\in\mathcal{S}}|A_1\cap A_2|<\infty$, then $\mathcal{S}$ is a separation system of $G$.
\end{lem}

\begin{prop}[{\cite[Section 3]{CHM}}]
Let $(T,\{V_t\}_{t\in VT})$ be a proper tree decomposition of $G$ with finite adhesions. For each edge $e=\{t_1,t_2\}\in {ET}$, let
\begin{align*}
    &A_{e,1}=\bigcup\{V_t\mid t\in VT,\ d_T(t,t_1)<d_T(t,t_2)\} \textup{ and }\\
    &A_{e,2}=\bigcup\{V_t\mid t\in VT,\ d_T(t,t_2)<d_T(t,t_1)\}.
\end{align*}
Then the family $\mathcal{S}_T=\{\{A_{e,1},A_{e,2}\}\mid e\in {ET}\}$ is a separation system of $G$, and the map $e\in ET\mapsto \{A_{e,1},A_{e,2}\}\in\mathcal{S}_T$ is bijective.

Conversely, for every separation system $\mathcal{S}$ of $G$, there exists is a unique proper tree decomposition $(T,\{V_t\}_{t\in VT})$ of $G$ with finite adhesions which induces $\mathcal{S}$. This is called the tree decomposition of $G$ associated with the separation system $\mathcal{S}$.
\end{prop}

Now let $(X,G)$ be a Borel graph. A separation of $G$ means a separation of a component of $G$. Since a separation $(A_1,A_2)$ of $G$ is determined by only by $A_1\cap A_2$, $\partial_\mathrm{ov}^GA_1$ and $\partial_\mathrm{ov}^GA_2$, the set of separations of $G$ admits a standard Borel space structure, which is denoted by $\mathrm{Sep}(G)$.

\begin{dfn}
Let $\{G_\lambda\}_{\lambda\in\Lambda}$ be the set of components of $G$. A \textit{Borel tree decomposition} of $G$ is a union $(T,\{V_t\}_{t\in VT})=(\bigsqcup_{\lambda\in\Lambda}T_\lambda,\bigsqcup_{\lambda\in\Lambda}\{V_t\}_{t\in VT_\lambda})$ of tree decompositions $(T_\lambda,\{V_t\}_{t\in VT_\lambda})$ of $G_\lambda$ with finite adhesions, such that the union $\bigsqcup_{\lambda\in\Lambda}\mathcal{S}_\lambda$ of the separation systems $\mathcal{S}_\lambda$ induced by $T_\lambda$ is a Borel subset of $\mathrm{Sep}(G)$.
\end{dfn}

Let $(T,\{V_t\}_{t\in VT)}$ be a Borel tree decomposition. Define
\begin{align*}
    &Y_T=\{(x,t)\in X\times VT\mid x\in V_t\}, \\
    &\mathcal{P}_T=\{((x,t),(y,t'))\in Y_T\times Y_T\mid (x,y)\in\mathcal{R}_G\}\textup{ and }\\
    &\mathcal{Q}_T=\{((x,t),(y,t'))\in Y_T\times Y_T\mid t=t'\} \subset \mathcal{P}_T.
\end{align*}

Note that $V_T$ is not a standard Borel space in general.
However, $Y_T$ admits a standard Borel space structure such that the projection $\mathrm{proj}_X:Y_T\to X$ is a Borel map, and $\mathcal{P}_T$ and $\mathcal{Q}_T$ are countable Borel equivalence relations on $Y_T$ (\cite[Proposition 3.1]{J}). Then we can take a Borel section $\iota:X\to Y_T$ of $\mathrm{proj}_X$, which gives a Borel embedding $\iota:(X,\mathcal{R}_G)\to(Y_T,\mathcal{P}_T)$.

Note that $ET$ admits a structure of a standard Borel space since it is identified with the Borel subset of $\mathrm{Sep}(G)$ induced by $T$. We can take a Borel map $e:ET\to X$ so that $e(\{t,t'\})\in V_t\cap V_{t'}$ for every $\{t,t'\}\in ET$ since every adhesion is finite. Let $\mathcal{T}$ be the equivalence relation on $Y_T$ generated by the Borel graph $T_1$ defined by
\begin{equation}
    ET_1=\{\{(e(\{t,t'\}),t),(e(\{t,t'\}),t')\}\subset Y_T\mid \{t,t'\}\in ET\}. \label{eq6}
\end{equation}
Note that each component of $T_1$ is isomorphic to a subgraph of $T$, and thus it is a tree. Hence, the relation $\mathcal{T}$ is treeable. By construction, the following holds:

\begin{lem}
We have $\mathcal{P}_T=\mathcal{Q}_T\ast \mathcal{T}$.
\end{lem}

Since treeability is preserved by taking free products, we have the following:

\begin{thm}[{\cite[Theorem 3.4]{J}}]\label{thmJ}
If $\mathcal{Q}_T$ is (measure-)treeable, then so is $\mathcal{R}_G$.
\end{thm}

Now suppose that $G$ is pmp on $(X,\mu)$.   
By Remark \ref{rmk1}, there exists a unique $\mathcal{P}_T$-invariant $\sigma$-finite measure $\nu_T$ on $Y_T$ such that $\iota:(X,\mu,\mathcal{R}_G)\to(Y_T,\nu_T,\mathcal{P}_T)$ is an embedding. Then by Lemma \ref{lem14}:

\begin{lem}
If $(Y_T,\nu_T,\mathcal{Q}_T)$ is sofic, then so is $(X,\mu,\mathcal{R}_G)$.
\end{lem}

\subsection{Torsos of tree decompositions}

Let $G$ be a connected graph and $(T,\mathcal{V}=\{V_t\}_{t\in VT})$ a tree decomposition of $G$.

\begin{dfn}
 For $t\in VT$, the \textit{torso} on $V_t$ is the graph 
\begin{equation*}
    H_t=G[V_t]\cup \bigcup_{\{t,t'\}\in ET} \{\{x,y\}\mid x,y\in V_t\cap V_{t'}\}.
\end{equation*}
\end{dfn}

\begin{lem}\label{lem6}
For every $t\in VT$, the torso on $V_t$ is connected.
\end{lem}

\begin{rmk}
Lemma \ref{lem6} implies that if the induced graph of $G$ on every adhesion is connected, then $G[V_t]$ is connected for every $t\in VT$. 
\end{rmk}

\begin{dfn}
Let $n$ be positive integer. A connected graph is said to be $n$-\textit{connected} if it is connected after removing any $n-1$ vertices (and the incident edges).
\end{dfn}

\begin{exm}[\textit{Decomposition into $1$-blocks}]
Let $\mathcal{S}_1$ be the set of $1$-separations of $G$ that are nested with all $1$-separations of $G$. Then each torso of the tree decomposition associated with $\mathcal{S}_1$ is called a $1$-\textit{block} of $G$, which is either a single vertex, a single edge or a $2$-connected subgraph of $G$.
\end{exm}

\begin{exm}[\textit{Tutte decomposition}]
Suppose that $G$ is $2$-connected, and let $\mathcal{S}_2$ be the set of $2$-separations of $G$ that are nested with all $2$-separations of $G$. The tree decomposition associated with $\mathcal{S}_2$ is called the \textit{Tutte decomposition} of $G$. Every torso of Tutte decomposition is called a $2$-\textit{block} of $G$, which is either a single edge, a polygon or a $3$-connected graph.

Note that every torso of the Tutte decomposition of $G$ is a minor of $G$. In particular, if $G$ is planar, so is every torso. 
\end{exm}

Again let $(X,G)$ be a Borel graph and $(T,\{V_t\}_{t\in VT})$ be a Borel tree decomposition of $G$. Take $Y_T,\mathcal{P}_T$ and $\mathcal{Q}_T$ as in the last subsection. For each $t\in VT$, let $H_t$ be the torso of $T$ on $V_t$. Let $H$ be the Borel graph on $Y_T$ defined by
\begin{equation*}
    EH=\{\{(x,t),(y,t)\}\subset Y_T\mid  t\in VT,\ \{x,y\}\in EH_t\}.
\end{equation*}
Then $H$ generates the relation $\mathcal{Q}_T$. We call $H$ the \textit{torsos Borel graph} associated with the Borel tree decomposition $T$. By the arguments in the last subsection, we have the following:

\begin{lem}\label{lem16}
There exists a treeable relation $\mathcal{T}$ on $Y_T$ such that $(X,\mathcal{R}_G)$ is embedded into $(Y_T,\mathcal{R}_H\ast\mathcal{T})$.
\end{lem}

\begin{cor}\label{cor2}
Let $(X,\mu,G)$ be a pmp planar graph. Then there exists a measure-preserving planar graph $(Y,\nu,H)$ such that
\begin{enumerate}
    \item every component of $H$ is either a single vertex, a single edge, a polygon or a $3$-connected planar graph, and
    \item there exists a treeable relation $\mathcal{T}$ on $Y$ such that $(X,\mu,\mathcal{R}_G)$ is embedded into $(Y,\nu,\mathcal{R}_H\ast \mathcal{T})$.
\end{enumerate}
In particular, $\mathcal{R}_G$ is $\mu$-sofic if and only if $\mathcal{R}_H$ is $\nu$-sofic. 
\end{cor}

\begin{proof}
Let $H_1$ be the torsos Borel graph associated with the Borel tree decomposition $T_1$ of $G$ into the blocks of $G$. By Lemma \ref{lem16}, $(X,\mathcal{R}_G)$ is embedded into $(Y_{T_1},\mathcal{R}_{H_1}\ast \mathcal{T}_1)$ for some treeable relation $\mathcal{T}_1$. Also, each component of $H_1$ is either a single vertex, a single edge, or a $2$-connected planar graph.

Let $T_2$ be the Borel tree decomposition of $H_1$ such that the decompositions on $2$-connected components are the Tutte decompositions and those on the other components are trivial. Let $H_2$ be the torsos graph associated with $T_2$. Then again $(Y_{T_1},\mathcal{R}_{H_1})$ is embedded into $(Y_{T_2,}\mathcal{R}_{H_2}\ast \mathcal{T}_2)$ for some treeable relation $\mathcal{T}_2$. Moreover, each component of $H_2$ is either a single vertex, a single edge, a cycle or a $3$-connected planar graph.

By regarding $Y_{T_1}$ as a subset of $Y_{T_2}$, let $\mathcal{T}_1'$ be the equivalence relation on $Y_{T_2}$ defined by
\begin{equation*}
    \mathcal{T}_1'=\mathcal{T}_1\cup\{(y,y)\mid y\in Y_{T_2}\setminus Y_{T_1}\}.
\end{equation*}
Then $(Y_{T_1},\mathcal{R}_{H_1}\ast\mathcal{T}_1)$ is embedded into $(Y_{T_2},\mathcal{R}_{H_2}\ast \mathcal{T}_1'\ast\mathcal{T}_2)$. Since $(X,\mu,\mathcal{R}_G)$ is pmp, there exists a measure $\nu$ on $Y_{T_2}$ such that $\mathcal{R}_{H_2}\ast \mathcal{T}_1'\ast\mathcal{T}_2$ is measure-preserving on $(Y_{T_2},\nu)$ and $(X,\mu,\mathcal{R}_G)$ is embedded into $(Y_{T_2},\nu,\mathcal{R}_{H_2}\ast \mathcal{T}_1'\ast\mathcal{T}_2)$ by Remark \ref{rmk1}.
\end{proof}

\subsection{The algebra of cuts and irreducibility}
In this subsection, we explain a construction of a ``canonical" family of cuts of a graph established by Dicks--Dunwoody \cite{DD}, which gives a ``nice" tree decomposition of the graph. We follow the formulation of Chen's unpublished note \cite{C}, where the notion of irreducibility of cuts is introduced. 

Let $G$ be a connected graph.

\begin{dfn}
\begin{enumerate}
    \item A \textit{cut} of $G$ is a finite subset $C\subset EG$ such that $C=\partial_\mathrm{e}^GA$ for a subset $A$ of $VG$.
    Let $\mathrm{Cut}(G)$ denote the set of cuts of $G$. By identifying $C\in \mathrm{Cut}(G)$ with the characteristic function $1_C\in {\mathbb{Z}_2}^{\oplus EG}$, the set $\mathrm{Cut}(G)$ is regarded as an algebra over $\mathbb{Z}_2$.
    \item Two cuts $C,D\in\mathrm{Cut}(G)$ are said to be \textit{nested} if there exists $A,B\subset VG$ with $C=\partial^G_\mathrm{e}A$ and $D=\partial^G_\mathrm{e}B$ such that either $A\subset B$ or $B\subset A$.
    \item A non-empty cut $C\in\mathrm{Cut}(G)$ is said to be \textit{minimal} if there is no non-empty $D\in\mathrm{Cut}(G)$ such that $D\subsetneq C$. A cut $C$ of $G$ is minimal if and only if the graph $G\setminus C$ consists of exactly two connected components.
\end{enumerate}
\end{dfn}

\begin{dfn}[{\cite[section 3]{C}}]
Let $\mathcal{A}\subset\mathrm{Cut}(G)$ be a subalgebra. For every $C\in\mathcal{A}$, we define the \textit{irreducibility} of $C$ in $\mathcal{A}$ inductively on $|C|$. If $|C|=0$, then $C$ is not irreducible. If the irreducibility of all $D\in\mathcal{A}$ with $|D|<|C|$ is defined, then let $\varepsilon C$ be the set of $C'\in\mathcal{A}$ satisfying
\begin{enumerate}
    \item $|C'|\leq |C|$,
    \item $C'$ is minimal,
    \item $C'$ is non-nested with $C$, and
    \item $C'$ is nested with all irreducible $D\in\mathcal{A}$ with $|D|<|C|$.
\end{enumerate}
Then $C$ is irreducible if it is not contained in the subalgebra generated by the set of $D\in\mathcal{A}$ such that $(|D|,|\varepsilon D|)$ is strictly less than $(|C|,|\varepsilon C|)$ in the lexicographic order.
\end{dfn}

\begin{rmk}\label{rmk2}
By definition, if $C\in \mathrm{Cut}(G)$ is irreducible in $\mathcal{A}$, then $C$ is not a sum of two cuts $D_1$ and $D_2$ in $\mathcal{A}$ such that $|D_1|,|D_2|<|C|$.
\end{rmk}

\begin{thm}[{\cite[Theorem 9]{C}}]\label{thmC}
Let $\mathcal{A}\subset\mathrm{Cut}(G)$ be a subalgebra. Suppose that every cut in $\mathcal{A}$ is a sum of minimal cuts in $\mathcal{A}$. Then the set $\mathcal{A}^\mathrm{irr}$ of irreducible cuts in $\mathcal{A}$ is pairwise-nested and generates $\mathcal{A}$. Moreover, every $C\in\mathcal{A}$ is contained in the subalgebra generated by the set of irreducible cuts in $\mathcal{A}$ of size at most $|C|$.
\end{thm}

\begin{rmk}
Let $(X,G)$ be a Borel graph and $\{G_\lambda\}_{\lambda\in\Lambda}$ be the set of components of $G$. Let $\mathrm{Cut}(G)=\bigsqcup_{\lambda\in\Lambda}\mathrm{Cut}(G_\lambda)$ be the set of cuts of components of $G$, which admits a structure of a standard Borel space. If $\mathcal{A}=\bigsqcup_{\lambda\in\Lambda}\mathcal{A}_\lambda$ is Borel subset of $\mathrm{Cut}(G)$ such that $\mathcal{A}_\lambda$ is a subalgebra of $\mathrm{Cut}(G_\lambda)$, then the union $\mathcal{A}^\mathrm{irr}=\bigsqcup_{\lambda\in\Lambda}\mathcal{A}_\lambda^\mathrm{irr}$ is a Borel subset of $\mathcal{A}$, where $\mathcal{A}_\lambda^\mathrm{irr}$ is the set of irreducible cuts in $\mathcal{A}_\lambda$. This is because the definition of irreducibility is ``descriptive" enough, as verified in \cite[section 4]{C}.
\end{rmk}

\section{planar complexes}\label{sec4}

Let $\mathbb{S}=\{r\in\mathbb{R}^3\mid \lVert r\rVert=1\}$ be the $2$-dimensional sphere.

\subsection{Tracks on planar complexes}

\begin{dfn}
A \textit{planar complex} is a $2$-dimensional locally-finite simplicial complex such that every connected component can be topologically embedded into $\mathbb{S}$. A $1$-dimensional planar complex is called a \textit{planar graph}.
\end{dfn}

\begin{dfn}
Let $\Sigma$ be a planar complex.
\begin{enumerate}
    \item An edge of $\Sigma$ is called a boundary edge if it is incident to at most one face.
    \item A vertex $x$ of $\Sigma$ is said to be non-singular if the induced complex $\Sigma[\mathrm{N}_\Sigma(x)\cup\{x\}]$ is homeomorphic to the closed disk in the plane. This is equivalent to that an open neighborhood of $x$ in the topological realization of $\Sigma$ is homeomorphic to $\mathbb{R}^2$ or $\mathbb{R}\times[0,\infty)$. We say that $\Sigma$ is without singular vertices if every vertex is non-singular.
\end{enumerate}
\end{dfn}

Let $\Sigma$ be a connected planar complex without singular vertices. Note that every edge is incident to either one or two faces.

\begin{dfn}
Let $\widehat{\Sigma}$ be the adjacency graph on the set $F\Sigma$ of faces of $\Sigma$. If adjacent faces $f_1,f_2\in F\Sigma$ shares an edge $e$, then the edge $\{f_1,f_2\}\in E\widehat{\Sigma}$ is denoted by $\widehat{e}$. By abusing the notation, we sometimes mean an orientation of this edge by $\widehat{e}$ as well.
\end{dfn}

Let $\mathcal{A}_\Sigma$ denote the subalgebra of cuts of the $1$-skeleton $\Sigma^1$ of $\Sigma$ which does not contain any boundary edge of $\Sigma$. Then the following is clear:

\begin{lem}\label{lem17}
The map $\tau:\{e_1,...,e_n\}\in \mathcal{A}_\Sigma\mapsto \delta_{\widehat{e}_1}+\delta_{\widehat{e}_2}+\cdots+\delta_{\widehat{e}_n}\in \mathrm{H}_1(\widehat{\Sigma},\mathbb{Z}_2)$ is well-defined and is an isomorphism. Moreover, $C\in \mathcal{A}_{\Sigma}$ is a minimal cut if and only if $\tau(C)$ is a simple cycle.
\end{lem}

\begin{dfn}
A minimal cut in $\mathcal{A}_\Sigma$ is called a \textit{track} on $\Sigma$. A track on $\Sigma$ is said to be \textit{irreducible} if the cut it is irreducible in $\mathcal{A}_\Sigma$.
\end{dfn}

Fix a positive integer $n$. Let $\mathcal{C}_n$ be the set of irreducible tracks on $\Sigma$ of size at most $n$. By Theorem \ref{thmC}, every cut in $\mathcal{A}_\Sigma$ of size at most $n$ is a sum of cuts in $\mathcal{C}_n$. By Lemma \ref{lem17}, we have the following:

\begin{lem}\label{lem18}
An element of $\mathrm{H}_1(\widehat{\Sigma},\mathbb{Z}_2)$ of size at most $n$ is a sum of simple cycles in $\tau(\mathcal{C}_n)=\{\tau(C)\mid C\in\mathcal{C}_n\}$.
\end{lem}

Let $\Sigma'$ be the barycentric subdivision of $\Sigma$. A $\Sigma$-path $e_1e_2\cdots e_n$ is identified with the $\Sigma'$-path $(x_0,e_1)(e_1,x_1)\cdots(x_{n-1},e_n)(e_n,x_n)$, where $e_i=(x_{i-1},x_i)$. Also, a $\widehat{\Sigma}$-path $\widehat{e}_1\widehat{e}_2\cdots \widehat{e}_n$ is identified with the $\Sigma'$-path $(f_0,e_1)(e_1,f_1)\cdots (f_{n-1},e_n)(e_n,f_n),$
where $\widehat{e}_i=(f_{i-1},f_i)$.

\begin{lem}\label{lem19}
Suppose that $\Sigma^1$ has degrees bounded by $N$. Let $p$ be a closed $\Sigma$-path based at $x\in V\Sigma$. Then there exists a closed $\widehat{\Sigma}$-path $q$ based at $f\in F\Sigma$ incident to $x$ such that $\mathrm{len}_{\widehat{\Sigma}}(q)\leq N\mathrm{len}_\Sigma(p)$ and $p$ is homotopic to $e'qe'^{-1}$ in $\Sigma'$, where $e'=(x,f)\in \overrightarrow{E}\Sigma'$. In particular for every $c\in \mathrm{H}_1(\Sigma^1,\mathbb{Z}_2)$, there exists a $\widehat{c}\in\mathrm{H}_1(\widehat{\Sigma},\mathbb{Z}_2)$ such that $|\widehat{c}|\leq N |c|$ and $[c]=[\widehat{c}]\in\mathrm{H}_1(\Sigma',\mathbb{Z}_2)$.
\end{lem}

\begin{proof}
Set $p=e_1e_2\cdots e_n$, where $e_i=(x_i,x_{i+1})$ with $x=x_1=x_{n+1}$. For each $i=1,...,n$, take $f_i\in F\Sigma$ incident to $e_i$, and set $f_0=f_n$. Then both $f_{i-1}$ and $f_i$ are incident to $x_i$, and thus there exists a simple $\widehat{\Sigma}$-path $q_i$ from $f_{i-1}$ to $f_i$ such that $q_i$ goes through only $\Sigma$-faces incident to $x_i$ since $\Sigma$ has no singular vertices. Note that $\mathrm{len}_{\widehat{\Sigma}}(q_i)\leq N$ for every $i$. Then the $\widehat{\Sigma}$-path $q=q_1q_2\cdots q_n$ based at $f_0$ satisfies the desired property.
\end{proof}

By Lemmas \ref{lem18} and \ref{lem19}, we have:

\begin{cor}\label{cor1}
Suppose that $\Sigma^1$ has degrees bounded by $N$, and $\mathrm{H}_1(\Sigma,\mathbb{Z}_2)$ is generated by a family of simple cycles on $\Sigma$ of size at most $n$. Then $\mathrm{H}_1(\Sigma',\mathbb{Z}_2)$ is generated by $\tau(\mathcal{C}_{Nn})$.
\end{cor}

Now let $\Omega_n$ be the polygonal complex (i.e., the boundary of faces are can be polygons) obtained by attaching to $\Sigma'$ a face along $\tau(C)$ for every $C\in\mathcal{C}_n$. Note that $\Omega_n$ might not be planar any more.

\begin{prop}\label{lem3}
If $p$ is a closed ${\widehat{\Sigma}}$-path of length at most $n$, then $p$ is contractible in $\Omega_n$.
\end{prop}

In particular, Proposition \ref{prop0} follows. We will prove this after preparation.

\begin{dfn}
Let $p=\widehat{e}_1\widehat{e}_2\cdots\widehat{e}_n$ be a $\widehat{\Sigma}$-path. Fix the cyclic order $[\cdot,\cdot,\cdot]$ on $\{1,...,n\}$ such that if $i<j<k$ then $[i,j,k]$.

Let $C$ be a track on $\Sigma$ and $G_1,G_2$ be the two components of $\Sigma^1\setminus\{C\}$.
We say that $p$ \textit{crosses $e_i$ over} $C$ if there exists $j\in\{1,...,n\}$ such that $e_i\in G_\alpha$ and $e_j\in G_\beta$ with $\alpha\neq \beta$ and $e_k\in C$ for all $k\in\{1,...,n\}\setminus\{i,j\}$ satisfying $[j,k,i]$. 

Then let $\chi_C(p)$ be the number of $i\in\{1,...,n\}$ such that $p$ crosses $e_i$ over $C$. For a family $\mathcal{C}$ of tracks on $\Sigma$, set $\chi_{\mathcal{C}}(p)=\sum_{C\in \mathcal{C}} \chi_C(p)$.
\end{dfn}

\begin{lem}\label{lem7}
Let $\mathcal{C}$ be a pairwise-nested family of tracks on $\Sigma$ of bounded size. Let $\Omega$ be the polygonal complex obtained by attaching faces to $\Sigma'$ along $\tau(C)$ for all $C\in\mathcal{C}$.
\begin{enumerate}
    \item If a simple closed $\widehat{\Sigma}$-path $p$ satisfies $\chi_{\mathcal{C}}(p)=0$ and $[p]=0\in \mathrm{H}_1(\Omega,\mathbb{Z}_2)$, then $p$ is contractible in $\Omega$.
    \item Let $p$ and $q$ be simple $\widehat{\Sigma}$-paths with $s(p)=s(q)$ and $t(p)=t(q)$ such that $q$ is a subpath of $\tau(C)$ for some $C\in\mathcal{C}$. If $\chi_{\mathcal{C}}(pq^{-1})=0$ and $[pq^{-1}]=0\in \mathrm{H}_1(\Omega,\mathbb{Z}_2)$, then the closed $\widehat{\Sigma}$-path $pq^{-1}$ is contractible in $\Omega$.
\end{enumerate} 
\end{lem}

This lemma will be proved in the next subsection since we need some facts on a tree decomposition of $\Omega$.

\begin{proof}[Proof of Proposition \ref{lem3}]
First of all, we may assume that $p$ is a simple closed path. We prove by the induction on $\chi_{\mathcal{C}_n}(p)$. Note that we have $[p]=0\in\mathrm{H}_1(\Omega_n,\mathbb{Z}_2)$ by Lemma \ref{lem18}. If $\chi_{\mathcal{C}_n}(p)=0$, then $p$ is contractible in $\Omega_n$ by Lemma \ref{lem7} (i). 

Suppose that $\chi_{\mathcal{C}_n}(p)>0$. Let $\{C_1,...,C_k\}=\{C\in\mathcal{C}_n\mid \chi_C(p)>0\}$. We may assume that $C_1$ corresponds to a leaf edge of the tree decomposition associated with $\{C_1,...,C_k\}$, that is, there exists a component $G_1$ of $\Sigma\setminus\{C_1\}$ which does not intersect $\bigcup_{i=2}^kC_i$. Set $p=\widehat{e}_1\widehat{e}_2\cdots\widehat{e}_m$ with $\widehat{e}_j=\{f_{j-1},f_{j}\}$ with $f_0=f_m$. By changing the base point of $p$, we may assume that $p$ crosses $e_1\in G_1$ over $C_1$. Let $G_2$ be the other component of $\Sigma^1\setminus\{C_1\}$ and let $j>1$ be the smallest integer such that $e_j \subset G_2$. Note that the simple $\widehat{\Sigma}$-cycle $\tau(C_1)$ goes through $f_0$ and $f_{j-1}$, and thus there are two simple $\widehat{\Sigma}$-paths $q$ and $r$ from $f_0$ to $f_{j-1}$ that are subpaths of $\tau(C)$ (one of which might be of length $0$). 

We claim that either $\mathrm{len}_{\widehat{\Sigma}}(q)$ or $\mathrm{len}_{\widehat{\Sigma}}(r)$ is at most $j-1$. Otherwise, we have 
\begin{align*}
    &\mathrm{len}_{\widehat{\Sigma}}(\widehat{e}_1\widehat{e}_{2}\cdots\widehat{e}_{j-1}q^{-1})=j-1+\mathrm{len}_{\widehat{\Sigma}}(q)< \mathrm{len}_{\widehat{\Sigma}}(r)+\mathrm{len}_{\widehat{\Sigma}}(q)\textup{ and } \\
    &\mathrm{len}_{\widehat{\Sigma}}(\widehat{e}_1\widehat{e}_{2}\cdots\widehat{e}_{j-1}r^{-1})=j-1+\mathrm{len}_{\widehat{\Sigma}}(r)< \mathrm{len}_{\widehat{\Sigma}}(q)+\mathrm{len}_{\widehat{\Sigma}}(r).
\end{align*}
Then however, we have 
\begin{equation*}
    \tau(C_1)=[rq^{-1}]=[\widehat{e}_1\widehat{e}_{2}\cdots\widehat{e}_{j-1}q^{-1}]+[\widehat{e}_1\widehat{e}_{2}\cdots\widehat{e}_{j-1}r^{-1}] \in \mathrm{H}_1(\widehat{\Sigma},\mathbb{Z}_2),
\end{equation*}
and this contradicts that $C_1$ is an irreducible track (Remark \ref{rmk2}). Hence we may assume that the closed path $p'=q\widehat{e}_j\widehat{e}_{j+1}\cdots\widehat{e}_m$ has length at most $m$. Set $q=\widehat{e'}_1\widehat{e'}_2\cdots\widehat{e'}_{j'}$. Note that $\{e_1',...,e_{j'}'\}\subset C_1$ since $q$ is a subpath of $\tau(C_1)$.

We show that $\chi_{\mathcal{C}_n}(p')<\chi_{\mathcal{C}_n}(p)$. First, we have $\chi_{C_1}(p')=\chi_{C_1}(p)-2$ by construction. Let $C\in\mathcal{C}_n\setminus\{C_1\}$, and it suffices to show that $\chi_{C}(p')\leq\chi_{C}(p)$. Let $H_1$ and $H_2$ be the two components of $\Sigma^1\setminus\{C\}$. Set
\begin{equation*}
    \{{l_1},{l_2},...,{l_{a}}\}=\{j\leq l\leq m\mid p' \textup{ crosses } e_l \textup{ over }C\}
\end{equation*}
with $l_1<l_2<\cdots<l_a$. Then it is clear that $p$ also crosses  $e_{l_2},e_{l_3},...,e_{l_a}$ over $C$. Since $\{e_1',...,e_{j'}'\}\subset C_1$ and $C_1$ is nested with $C$, the path $p'$ crosses at most one of $\{e_1',...,e_{j'}'\}$ over $C$, which implies that $\chi_C(p')\in \{a,a+1\}$. Case (i): If $\chi_C(p')=a$, then we have $e_{l_1}\in H_\alpha,e_{l_a}\in H_\beta$ with $\alpha\neq \beta$, and thus $p$ crosses at least one of $\{e_1,...,e_{l_1}\}$ over $C$. Hence $\chi_C(p)\geq a$. Case (ii): If $\chi_C(p')=a+1$, then we have $e_{l_1},e_{l_a}\in H_\alpha$ and there exists $e_{l'}'\in H_\beta$ with $1\leq l'\leq j'$ and $\alpha\neq \beta$. Then by the nested-ness, we have $C_1\subset C\cup H_\beta$, which implies that $G_1\subset H_\beta$ or $G_2\subset H_\beta$. If $G_2\subset H_\beta$, then $e_{l_a},e_{l_a+1},...,e_m\in H_\alpha\subset G_1$, which contradicts that $p$ crosses $e_1\in G_1$ over $C_1$. Hence we have $G_1\subset H_\beta$, and then $p$ crosses $e_1\in H_\beta$ and at least one of $\{e_2,e_3,...,e_{l_a}\}$ over $C$, which implies that $\chi_{C}(p)\geq a+1$.

Now we have verified that $\mathrm{len}_{\widehat{\Sigma}}(p')\leq \mathrm{len}_{\widehat{\Sigma}}(p)$ and $\chi_{\mathcal{C}_n}(p')<\chi_{\mathcal{C}_n}(p)$. Then by the induction hypothesis, $p'$ is contractible in $\Omega_n$, and in particular, $[p']=0\in\mathrm{H}_1(\Omega_n,\mathbb{Z}_2)$.

Set
\begin{equation*}
    p''=\widehat{e}_1\widehat{e}_2\cdots\widehat{e}_{j-1}q^{-1}=\widehat{e}_1\widehat{e}_2\cdots\widehat{e}_{j-1}\widehat{e'}_{j'}^{-1}\widehat{e'}_{j'-1}^{-1}\cdots\widehat{e'}_{1}^{-1}.
\end{equation*}
Since $p$ is homotopic to $p''p'$, it suffices to show that $p''$ is contractible in $\Omega_n$. Note that $[p'']=[p]+[p']=0\in\mathrm{H}_1(\Omega_n,\mathbb{Z}_2)$.
We claim that $\chi_{\mathcal{C}_n}(p'')=0$. Otherwise, there exists $C'\in \mathcal{C}_n\setminus\{C_1\}$ such that $\chi_{C'}(p'')>0$ since $\chi_{C_1}(p'')=0$. Since $e_1\in G_1$, the path $p''$ crosses only edges in $C_1\cup G_1$, and thus we have $C'\subset C_1\cup G_1$. Let $H_1'$ and $H_2'$ be the two components of $\Sigma^1\setminus C'$ such that $G_2\subset H_2'$. Since $\{e_1',e_2'...,e_{j'}'\}\subset C_1\subset C'\cup H_2'$, we have $\{e_1,...,e_{j-1}\}\cap H_1'\neq \varnothing$. Then, $p=\widehat{e}_1\widehat{e}_2\cdots\widehat{e}_m$ crosses over $C'$ since $e_j\in G_2\subset H_2'$. This implies that $C'\in\{C_2,...,C_k\}$, which contradicts that $G_1$ does not intersect $\bigcup_{i=2}^kC_i$. Hence we have $\chi_{\mathcal{T}_n}(p'')=0$. Then by Lemma \ref{lem7} (ii), $p''$ is contractible in $\Omega_n$.
\end{proof}

\subsection{Tree decompositions of planar complexes}

Let $\Sigma$ be a planar complex without singular vertices and $\Sigma'$ the barycentric subdivision of $\Sigma$.

For each track $C$ on $\Sigma$, we regard $\tau(C)$ as the simple cycle on $\Sigma'$ in this subsection. Also, let $A_{C,1}$ and $A_{C,2}$ be the vertex sets of two components of $\Sigma^1\setminus C$, and for each $i=1,2$, let $A_{C,i}'=A_{C,i}\cup\partial_\mathrm{ov}^{\Sigma'}A_{C,i}\subset V\Sigma'$. Note that $\{A_{C,1}',A_{C,2}'\}$ is a separation of $\Sigma'$ and the adhesion $A_{C,1}'\cap A_{C,2}'$ is the set of vertices which the $\Sigma'$-cycle $\tau(C)$ goes through.

Let $\mathcal{C}$ be a pairwise-nested family of tracks on $\Sigma$ of bounded size. Then $\{A_{C,1}',A_{C,2}'\}_{C\in\mathcal{C}}$ is a separation system of $\Sigma'$ by Lemma \ref{lem20}. Let $(T,\{V_t\}_{t\in VT})$ be the associated tree decomposition of $\Sigma'^1$. Now let $\Omega$ be the polygonal complex obtained by attaching faces to $\Sigma'$ along $\tau(C)$ for all $C\in \mathcal{C}$. For each $t\in VT$, let $\Omega[V_t]$ be the induced polygonal complex of $\Omega$ on $V_t$. Note that a face attached afterward is contained in some $\Omega[V_t]$. We call $(T,\{V_t\}_{t\in VT})$ \textit{the tree decomposition of $\Omega$ associated with} $\mathcal{C}$.

\begin{lem}\label{lem9}
The polygonal complex $\Omega[V_t]$ is connected and planar for every $u\in VT$.
\end{lem}

\begin{proof}
The connectivity is clear since the induced complex of $\Omega$ on every adhesion is connected.

Fix $t\in VT$. Since $\Sigma'$ is planar, so is $\Sigma'[V_t]$.
Let $i:\Sigma'[V_t]\to \mathbb{S}$ be an embedding. Let $\mathcal{C}_t$ be the set of $C\in\mathcal{C}$ contained in $\Sigma'[V_t]$. It suffices to show the following:
\begin{enumerate}
    \item For every $C\in \mathcal{C}$, the simple closed curve $i(\tau(C))$ bounds an open area $M_C\subset \mathbb{S}$ which does not intersect $i(\Sigma'[V_t])$.
    \item The family $\{M_C\}_{C\in\mathcal{C}_t}$ is pairwise-disjoint.
\end{enumerate}

For $C\in\mathcal{C}$, we may assume that $V_t\subset A_{C,1}'$ by reversing $A'_{C,1}$ and $A'_{C,2}$ if needed. By the definition of $\{A_{C,1}, A_{C,2}\}$, the simple closed curve $i(\tau(C))$ separates $\mathbb{S}$ into two open areas $M_{C,1}$ and $M_{C,2}$ such that $i(A_{C,i})\subset M_{C,i}$ for $i=1,2$. This implies that 
\begin{equation*}
    i(\Sigma'[V_t])\subset i(\Sigma'[A_{C,1}'])\subset \overline{M_{C,1}}=M_{C,1}\cup \partial M_{C,1},
\end{equation*}
and thus $i(\Sigma'[V_t])\cap M_{C,2}=\varnothing$. Claim (i) is proved. 

For $C,D\in\mathcal{C}_t$, we have $i(C),i(D)\subset i(\Sigma'[V_t])\subset \overline{M_{C,1}}\cap \overline{M_{D,1}}$ by the above. This implies that $M_{C,2}\cap M_{D,2}=\varnothing$, and claim (ii) is proved.
\end{proof}

Now we can prove Lemma \ref{lem7}.

\begin{proof}[Proof of Lemma \ref{lem7}]
Let $(T,\{V_t\}_{t\in VT})$ be the tree decomposition of $\Omega$ associated with $\mathcal{C}$.

(i) By $\chi_{\mathcal{C}}(p)=0$, there exists $t_0\in VT$ such that $p$ is contained in $\Omega[V_{t_0}]$, and thus it suffices to show that $p$ is contractible in $\Omega_n[V_{t_0}]$. Since $\pi_1(\Omega)=\bigast_{t\in VT}\pi_1(\Omega[V_t])$ holds, the assumption $[p]=0\in \mathrm{H}_1(\Omega,\mathbb{Z}_2)$ implies that $[p]=0\in \mathrm{H}_1(\Omega[V_{t_0}],\mathbb{Z}_2)$. In general, a simple closed curve on a planar surface is contractible if and only if the element of the first homology group represented by the curve is trivial. Since $\Omega[V_{t_0}]$ is planar by Lemma \ref{lem9}, $p$ is contractible in $\Omega[V_{t_0}]$.

(ii) Again, $pq^{-1}$ is contained in some $\Omega[V_{t_0}]$ and $[pq^{-1}]=0\in \mathrm{H}_1(\Omega[V_{t_0}],\mathbb{Z}_2)$. Moreover, we may assume $\tau(C)$ is also contained in $\Omega[V_{t_0}]$ by replacing $C$ with aother element in $\mathcal{C}$ if needed. If $p$ is a simple closed path, it is contractible in $\Omega$ by (i), and so is $pq^{-1}$. Hence suppose $s(p)\neq t(p)$. Let $r$ be the simple $\widehat{\Sigma}$-path such that $qr$ represents $\tau(C)$. We take a subdivision $\Omega[V_{t_0}]'$ of $\Omega[V_{t_0}]$ as follows: Add an edge $e$ connecting $s(p)$ and $t(p)$, which is orientated as $e=(s(p),t(p))$, and replace the face attached along $\tau(C)$ with two faces attached along the simple cycles represented by $qe^{-1}$ and $re^{-1}$, respectively. Then the simple closed path $pe^{-1}$ is homotopic to $pq^{-1}$ in $\Omega[V_{t_0}]'$, and thus we have $[pe^{-1}]=0\in\mathrm{H}_1(\Omega[V_{t_0}]',\mathbb{Z}_2)$. By the planarity of $\Omega[V_{t_0}]'$, the path $pe^{-1}$ is contractible in $\Omega[V_{t_0}]'$, which implies that $pq^{-1}$ is contractible in $\Omega[V_{t_0}]$.
\end{proof}

Note that for each $\{s,t\}\in ET$, the induced complex $\Omega[V_s\cap V_t]$ is homeomorphic to the closed disk. Also, $\Omega$ is homeomorphic to the polygonal complex obtained by gluing $\{\Omega[V_t]\}_{t\in VT}$ to each other, where the attachment maps are the identification of $\Omega[V_s\cap V_t]\subset \Omega[V_s]$ and $\Omega[V_s\cap V_t]\subset \Omega[V_t]$ for all $\{s,t\}\in ET$.

Let $\pi:\widetilde{\Omega}\to \Omega$ be the universal covering map. We will give a tree decomposition $(\widetilde{T},\{\widetilde{V}_t\}_{t\in V\widetilde{T}})$ of ${\widetilde{\Omega}}$.
By the van Kampen theorem, the fundamental group $\pi_1(\Omega)$ of $\Omega$ is identified with the free product $\bigast_{t\in VT}\pi_1(\Omega[V_t])$. Also, $\bigast_{t\in VT}\pi_1(\Omega[V_t])$ is a fundamental group of the graph of groups, where the underlining graph is $T$, the vertex group for each $t\in VT$ is $\pi_1(\Omega[V_t])$ and all edge groups are trivial. Let $\widetilde{T}$ be the universal covering tree of this graph of groups and let $\pi_1(\Omega)$ act on $\widetilde{T}$ through the above identification, with the quotient map $\rho:\widetilde{T}\to T$. Then there exists a $\pi_1(\Omega)$-equivariant map $t\in V\widetilde{T}\mapsto \widetilde{V}_t\subset V\widetilde{\Omega}$ such that $\pi|_{\widetilde{\Omega}[\widetilde{V}_t]}:\widetilde{\Omega}[\widetilde{V}_t]\to \Omega[V_{\rho(t)}]$ is the universal covering map for every $t\in V\widetilde{T}$, and $\widetilde{\Omega}[\widetilde{V}_s\cap \widetilde{V}_t]$ is a lift of $\Omega[V_{\rho(s)}\cap V_{\rho(t)}]$ for every $\{s,t\}\in E\widetilde{T}$. Then $(\widetilde{T},\{\widetilde{V}_t\}_{t\in V\widetilde{T}})$ is a tree decomposition of ${\widetilde{\Omega}}$.

Now we remove from $\widetilde{\Omega}$, the unique face in $\widetilde{\Omega}[\widetilde{V}_s\cap \widetilde{V}_t]$ for each $\{s,t\}\in E\widetilde{T}$, and then obtain a (normal) covering $\widetilde{\Sigma}'$ of $\Sigma'$. We call $\widetilde{\Sigma}'$ \textit{the covering of $\Sigma'$ associated with} $\mathcal{C}$. We will prove the following:

\begin{prop}\label{prop1}
The simplicial complex $\widetilde{\Sigma}'$ is planar and $\mathrm{H}_1(\widetilde{\Sigma}',\mathbb{Z}_2)$ is generated by the family of simple cycles on $\Sigma'$ of size at most $2\sup_{C\in\mathcal{C}}|C|$.
\end{prop}

We say that a cycle $C$ on a planar graph is \textit{facial} if the polygonal complex obtained by attaching a face to the graph along $C$ is planar.

\begin{lem}\label{lem8}
Let $\Sigma$ be a $2$-dimensional simplicial complex and $(T,\{V_t\}_{t\in VT})$ a tree decomposition of $\Sigma$. Suppose that for every $s\in VT$, the following conditions hold:
\begin{enumerate}
    \item $\Sigma[V_s]$ is a planar graph,
    \item if $t\in VT$ is adjacent to $s$, then $\Sigma[V_s\cap V_t]$ is $1$-dimensional and is a facial cycle in $\Sigma^1[V_t]$, and
    \item if $t,t'\in VT$ are distinct vertices adjacent to $s$, then the cycles $\Sigma[V_s\cap V_{t}]$ and $\Sigma[V_s\cap V_{t'}]$ are distinct from each other.
\end{enumerate}
Then $\Sigma$ is a planar complex.
\end{lem}

\begin{proof}
Since planarity is closed under taking inductive limits, we may assume that $T$ is a finite tree. For each $s\in VT$, let $i_s:\Sigma[V_s]\to \mathbb{S}_s\simeq \mathbb{S}$ be an embedding. Note that for $\{s,t\}\in ET$, the simple closed curve $i_s(\Sigma[V_s\cap V_t])$ bounds an open area $M_{s,t}\subset \mathbb{S}_s$ which is does not intersect $i_s(\Sigma[V_s])$. Moreover, if $\{s,t\},\{s,t'\}\in ET$ are distinct, then $M_{s,t}$ and $M_{s,t'}$ are disjoint with each other. Now we consider the connected sum $\#_{s\in VT} \mathbb{S}_s$ as follows: First, for each $s\in VT$, remove $M_{s,t}$ from $\mathbb{S}_s$ for all $t$ adjacent to $t$. Next, for each $\{s,t\}\in ET$, we glue $\mathbb{S}_s\setminus \bigcup_{t\sim s}M_{s,t}$ with $\mathbb{S}_t\setminus \bigcup_{s\sim t}M_{s,t}$ by the attachment map $f_{s,t}:\partial M_{s,t}\to \partial M_{t,s}$ so that $f_{s,t}\circ i_s|_{\Sigma[V_s\cap V_t]}=i_t|_{\Sigma[V_s\cap V_t]}$. Then by construction, there exists a embedding $i:\Sigma\to \#_{s\in VT} \mathbb{S}_s$ such that $i|_{\Sigma[V_s]}:\Sigma[V_s]\to \#_{s\in VT} \mathbb{S}_s$ coincides with $i_s:\Sigma[V_s]\to \mathbb{S}_s\setminus \bigcup_{t\sim s}M_{s,t}\hookrightarrow \#_{s\in VT} \mathbb{S}_s$ for every $s\in VT$. Since $\#_{s\in VT} \mathbb{S}_s^2$ is homeomorphic to $\mathbb{S}$, the lemma is proved.
\end{proof}

\begin{proof}[Proof of Proposition \ref{prop1}]
For every $t\in V\widetilde{T}$, the complex $\widetilde{\Omega}[\widetilde{V}_t]$ is planar since it is homeomorphic to a simply connected surface with boundary. Then by Lemma \ref{lem8}, $\widetilde{\Sigma}'$ is planar. Moreover, since $\widetilde{\Omega}$ is simply connected, $\mathrm{H}_1(\widetilde{\Sigma}',\mathbb{Z}_2)$ is generated by the boundaries of the faces in $\widetilde{\Omega}$. The length of these boundaries are bounded by $\sup_{C\in\mathcal{C}}|\tau(C)|=2\sup_{C\in\mathcal{C}}|C|$. Here recall that $\tau(C)$ is considered as a cycle on $\Sigma'$.
\end{proof}

\section{Planar Borel complexes}\label{sec5}

\subsection{Measure-treeability}
\begin{thm}[{\cite{CGMT}}]\label{fact1}
Let $(X,\Omega)$ be a planar Borel polygonal complex.
If $\mathrm{H}_1(\Omega,\mathbb{Z}_2)=0$, then $\mathcal{R}_\Omega$ is measure-treeable.
\end{thm}

\begin{proof}
The set of the boundaries of faces of $\Omega$ is a Borel $2$-basis of $\Omega^1$ in the sense of \cite[Definition 3.5]{CGMT}. By \cite[Theorem 3.6]{CGMT}, $\mathcal{R}_\Omega$ is measure-treeable.
\end{proof}

The following is inspired by Theorem \cite[Theorem 5.2]{J}:

\begin{thm} \label{fact2} 
Let $(X,\Sigma)$ be a bounded-degree planar Borel complex without singular vertices.
If $\mathrm{H}_1(\Sigma,\mathbb{Z}_2)$ is boundedly generated, then $\mathcal{R}_\Sigma$ is measure-treeable.
\end{thm}

\begin{proof}
Let $(X',\Sigma')$ be the barycentric subdivision of $(X,\Sigma)$. It suffice to show that $(X',\mathcal{R}_{\Sigma'})$ is measure-treeable. Suppose that $\Sigma^1$ has degrees bounded by $N$, and $\mathrm{H}_1(\Sigma,\mathbb{Z}_2)$ is generated by a family of cycles on $\Sigma$ of size at most $n$. Let $\mathcal{C}$ be the set of irreducible tracks on $\Sigma$ of size at most $Nn$. Let $\Omega$ be a Borel polygonal complex on $X'$ obtained by attaching faces to $\Sigma'$ along $\tau(C)$ for all $C\in \mathcal{C}$. Since $\mathrm{H}_1(\Sigma',\mathbb{Z}_2)$ is generated by $\tau(\mathcal{C})$ by Corollary \ref{cor1}, we have $\mathrm{H}_1(\Omega,\mathbb{Z}_2)=0$. Let $(T,\{V_t\}_{t\in VT})$ be the tree decomposition of $\Omega$ associated with $\mathcal{C}$. For every $t\in VT$, the complex $\Omega[V_t]$ is planar by Lemma \ref{lem9}, and satisfies $\mathrm{H}_1(\Omega[V_t],\mathbb{Z}_2)=0$.

Let $(Y_T,\mathcal{Q}_T)$ be the equivalence relation associated with the tree decomposition $T$ as in Section \ref{sec3.1}.
Then we can take a planar Borel polygonal complex $\Omega_T$ on $Y_T$ so that $\mathcal{R}_{\Omega_T}=\mathcal{Q}_T$ and the projection map $(x,t)\in Y_T\mapsto x\in X'$ induces the isomorphism $\Omega_T[[(x,t)]_\mathcal{Q}]\to \Omega[V_t]$ for every $(x,t)\in Y_T$. Hence we have $\mathrm{H}_1(\Omega_T,\mathbb{Z}_2)=0$.
Now $(Y_T,\mathcal{Q}_T)$ is measure-treeable by the previous theorem, and so is $(X',\mathcal{R}_{\Sigma'})$ by Theorem \ref{thmJ}.
\end{proof}

\subsection{Reduction to from Borel graphs to Borel complexes}
Fix an orientation of on the sphere $\mathbb{S}$. For a embedding of a finite star into $\mathbb{S}$, the cyclic order on the set of end vertices with respect to the center $o$ of the star is defined. We call this \textit{the clockwise order centered at} $o$ with respect to this embedding.
For a connected planar graph $G$, a \textit{spherical rotation}
for $G$ is a family $\{\omega(x)\}_{x\in VG}$ such that there exists an embedding $i:G\to\mathbb{S}$ such that $\omega(x)$ is a clockwise order on $\mathrm{N}_G(x)$ centered at $x$ with respect to $i$ for every $x\in VG$.

\begin{dfn}[{\cite[Definition 5.3]{J}}]
Let $(X,G)$ be a planar Borel graph. A \textit{Borel rotation system} of $G$ is a family $\{\omega(x)\}_{x\in X}$ such that
\begin{enumerate}
    \item for every component $G_0$ of $G$, the family $\{\omega(x)\}_{x\in VG_0}$ is a spherical rotation for $G_0$
    \item $(x,y,z,w)\in\{(x,y,z,w)\mid x\in X,\ y,z,w\in\mathrm{N}_G(x) \}\mapsto 1_{\omega(x)}(y,z,w)$ is a Borel map.
\end{enumerate}
\end{dfn}

\begin{lem}[{\cite[Proposition 5.5]{J}}]\label{lem11}
Let $(X,G)$ be a planar Borel graph with $3$-connected components. Then there exist a planar Borel graph $(Y,H)$ with a Borel rotation system, and a $2$-to-$1$ Borel map $\pi:Y\to X$ such that $\pi:(Y,H)\to (X,G)$ is a component-wise isomorphism.
\end{lem}

\begin{lem}\label{lem12}
Let $(X,G)$ be a planar Borel graph with $3$-connected components. Suppose that $G$ admits a Borel rotation system $\{\omega(x)\}_{x\in X}$. Then there exists a planar Borel complex $(Y,\Sigma)$ without singular vertices such that $(X,\mathcal{R}_G)$ is embedded into $(Y,\mathcal{R}_\Sigma)$.
\end{lem}

\begin{proof}
We define a Borel map $\Phi:\overrightarrow{E}G\to \overrightarrow{E}G$. For $e=(x,y)\in \overrightarrow{E}G$, let $\Phi(e)=(y,z)$ where $z\in\mathrm{N}_G(y)\setminus\{x\}$ is the unique element such that there is no $w\in\mathrm{N}_G(y)$ satisfying $(z,w,x)\in \omega(y)$. Note that $G$ has no vertex of degree less than $2$ since $G$ is $3$-connected, and thus $\Phi$ is well-defined and invertible. For each component $G_0$ of $G$ and an edge $e\in E\Sigma_0$, the path $\cdots\Phi^{-1}(e)\cdot e\cdot \Phi(e)\cdots$ goes in the most clockwise direction with respect to an embedding $i:\Sigma_0\to\mathbb{S}$. Thus the $\Phi$-orbit of $e$ is either a simple cycle or a bi-infinite line, which is sent by $i$ to the boundary of a component of $\mathbb{S}\setminus i(G_0)$.

Let $Z$ be a standard Borel space with a Borel bijection $\Psi:Z\to  \overrightarrow{E}G$ and set $Y=X\sqcup Z$. Now let $\Sigma$ be the smallest complex on $Y$ containing the set
\begin{align*}
    &\{\{x,y,z\}\subset Y\mid z\in Z,\ \Psi(z)=(x,y)\}\cup\\
    &\{\{x,y,z\}\subset Y\mid y,z\in Z,\  x=s(\Psi(y))=t(\Psi(z)),\ \Phi(\Psi(y))=\Psi(z)\}.
\end{align*}
We show that $\Sigma$ is a planar complex. Let $G_0$ be a component of $G$ and take an embedding $i:G_0\to\mathbb{S}$ as above. To every finite $\Phi$-orbit on $G_0$, we attach an annulus $[0,1]\times \mathbb{T}$ along $\{0\}\times\mathbb{T}$, and to every infinite $\Phi$-orbit on $G_0$, we attach a bi-infinite band $[0,1]\times \mathbb{R}$ along $\{0\}\times\mathbb{R}$. Then we obtain the topological realization of a component $\Sigma_0$ of $\Sigma$. By the property of $\Phi$-orbits, we can extend $i$ to an embedding of $\Sigma_0\to\mathbb{S}$. Hence $\Sigma$ is planar.

Moreover, $\Sigma$ has no singular vertex since in this topological realization, an open neighborhood of every vertices in $X$ is homeomorphic to $\mathbb{R}$, and an open neighborhood of every vertices in $Z$ is homeomorphic to $\mathbb{R}\times [0,\infty)$. Finally, it is clear that the inclusion map $(X,\mathcal{R}_G)\hookrightarrow (Y,\mathcal{R}_\Sigma)$ is an Borel embedding.
\end{proof}

\begin{rmk}
In fact, by an argument similar to the proof of \cite[Corollary 5.7]{J}, Lemma \ref{lem12} holds without assuming the existence of a Borel rotation system. We do not need this for our purpose.
\end{rmk}

\subsection{Proof of Theorems}
Finally, we prove Theorems \ref{thm2'} and \ref{thm3}.

\begin{thm}\label{thm2}
Let $(X,\mu,\Sigma)$ be a pmp bounded-degree planar complex without singular vertices. Then there exist a sequence of pmp planar complexes $(X_n,\mu_n,\Sigma_n)$ such that $\mathcal{R}_{\Sigma_n}$ is $\mu_n$-treeable for every $n$, and a sequence of covering factor maps $\pi_n:(X_n,\mu_n,\Sigma_n)\to (X,\mu,\Sigma)$ which is an asymptotic extension. In particular, $\mathcal{R}_{\Sigma}$ is $\mu$-sofic.
\end{thm}

\begin{proof}
Fix an integer $n\geq 3$. Let $\mathcal{C}_n$ be the family of irreducible tracks on $\Sigma$ of size at most $n$. Let $(X',\Sigma')$ be the barycentric subdivision of $(X,\Sigma)$, and we extend $\mu$ to the finite measure on $X'$ so that $\mathcal{R}_{\Sigma'}$ is measure-preserving. Then let $(X',\Omega_n)$ be the Borel polygonal complex defined by attaching to $\Sigma'$ a face along $\tau(C)$ for every $C\in \mathcal{C}_n$. Let $\Pi_1(\Omega_n)$ be the fundamental groupoid of $\Omega_n$ and $\varepsilon_n: \Pi_1(\Omega_n)\to \mathcal{R}_{\Sigma'},\ g\mapsto (s(g),t(g))$ the canonical projection (Example \ref{exm2}).

Now take a principal extension $\rho_n:(X'_n,{\mu_n},\mathcal{R}_n)\to (X',{\mu},\Pi_1(\Omega_n))$, and set $\pi_n=\varepsilon_n\circ\rho_n$. Let $\Sigma_n'$ be the Borel complex on $X_n'$ such that $\mathcal{R}_{\Sigma_n}=\mathcal{R}_n$ and $\pi_n:(X_n',\mu_n,\Sigma_n')\to (X',\mu,\Sigma')$ is a covering factor map.
By construction, each component of $\Sigma_n'$ is the covering of a component of $\Sigma'$ associated with the restriction of $\mathcal{C}_n$ on this component. Hence $\Sigma_n'$ is a measure-preserving bounded-degree planar complex without singular vertices, and $\mathrm{H}_1(\Sigma_n',\mathbb{Z}_2)$ is boundedly generated by Proposition \ref{prop1}. Then by Theorem \ref{fact2}, $\mathcal{R}_{\Sigma_n}$ is $\mu_n$-treeable.

We will show that the sequence $\pi_n:(X_n',\mu_n,\Sigma_n')\to (X',\mu,\Sigma')$ is an asymptotic extension. Fix $k\in\mathbb{Z}_{\geq 1}$. By Proposition \ref{lem3} and Lemma \ref{lem19}, there exists $n_0$ such that any closed $\Sigma'$-path of length at most $k$ is contractible in $\Omega_{n_0}$. Then for $n\geq n_0$ and $\varphi^1,\varphi^2,...,\varphi^k\in[[\Sigma^1]]$, the lifts $\varphi^1_n,\varphi^2_n,...,\varphi^k_n\in[[\Sigma_n^1]]$ satisfies $\tau_{\mu_n}(\varphi^1_n\cdots\varphi^k_n)=\tau_\mu(\varphi^1\cdots\varphi^k)$ since if $\varphi^1\cdots\varphi^kx=x$, then $\varphi^1_n\cdots\varphi^k_ny=y$ for every $y\in\pi_n^{-1}(x)$. This implies that $\pi_n:(X_n',\mu_n,\Sigma_n')\to (X',\mu,\Sigma')$ is an asymptotic extension.

Finally, $(X_n,\mu_n,\Sigma_n)$ be the pmp planar complex such that its barycentric subdivision is $(X_n',\mu_n,\Sigma_n')$. Then it is clear that $\mathcal{R}_{\Sigma_n}$ is $\mu_n$ treeable for every $n$, and $\pi_n|_{X_n}:(X_n,\mu_n,\Sigma_n)\to (X,\mu,\Sigma)$ is a sequence of covering factor maps which is an asymptotic extension.

In particular, $\mathcal{R}_\Sigma$ is $\mu$-sofic by Lemma \ref{lem1}.
\end{proof}

\begin{cor}\label{cor3}
Let $G$ be a pmp planar graph on $(X,\mu)$. Then $\mathcal{R}_G$ is $\mu$-sofic.
\end{cor}

\begin{proof}
By Corollary \ref{cor2}, we may assume that $(X,\mu,G)$ is a measure-preserving planar graph with $3$-connected components. Then by Lemmas \ref{lem11} and \ref{lem12}, there exist a measure-preserving planar complex $(Y,\nu,\Sigma)$ without singular vertices such that a $2$-to-$1$ extension of $(X,\mu,\mathcal{R}_G)$ is embedded into $(Y,\nu,\mathcal{R}_\Sigma)$. Hence it suffices to show that $\mathcal{R}_\Sigma$ is $\nu$-sofic. Since soficity is preserved by taking increasing unions, we may assume that $\nu(Y)<\infty$ and $\Sigma$ has bounded degrees. Then by Theorem \ref{thm2}, $\mathcal{R}_\Sigma$ is $\nu$-sofic.
\end{proof}

\section*{Acknowledgments}
I would like to thank Héctor Jardón-Sánchez for sharing many questions and ideas, which led me to do this work. I am also grateful to Antoine Poulin, Oriol Sol\'{e}-Pi and Anush Tserunyan for intensive discussions. This work was supported by JSPS KAKENHI Grant Number 23KJ0653, and the WINGS-FMSP program in the University of Tokyo.

\footnotesize
\textsc{Research Institute for Mathematical Sciences, Kyoto University, Kyoto 606-8502, Japan.}\par
\texttt{ishikura@kurims.kyoto-u.ac.jp}

\end{document}